\documentclass[12pt,a4paper,reqno]{amsart}
\usepackage{preemble, relsize, enumerate}

\usepackage{times,amsfonts,amsmath,amstext,amsbsy,amssymb,
  amsopn,amsthm,upref,eucal, amscd}
\usepackage[colorlinks=true]{hyperref}
\usepackage[T1]{fontenc}
\usepackage[backgroundcolor=white]{todonotes}
\usepackage{multirow}
\usepackage{color}
\usepackage{graphicx}
\usepackage{accents}

\date{\today}

\author{Stephen Cantrell}
\address{Department of Mathematics, 
University of Warwick,
Coventry, CV4 7AL, UK}
\email{stephen.cantrell@warwick.ac.uk}

\author{Mark Pollicott}
\address{Department of Mathematics, 
University of Warwick,
Coventry, CV4 7AL, UK}
\email{masdbl@warwick.ac.uk}

\date{\today. 
\\
2020 \textit{Mathematics Subject Classification: 30F10, 60F05, 11N45}. 
\\
\textit{Key words and phrases: Translation surfaces, saddle connections, counting, limit laws}.
}

\renewcommand{\E}{\mathcal E}
\renewcommand{\G}{\mathcal G}
\renewcommand{\P}{\mathcal P}
\renewcommand{\Re}{\mathfrak{Re}}

\title[Counting statistics for geodesics on flat surfaces]{Counting statistics for geodesics on flat surfaces}

\begin{document}
\maketitle

\begin{abstract}
We study counting limit laws that compare length functions on infinite graphs. We then apply these results to flat surfaces to obtain a statistical comparison between the geometric length and the number of singularities visited by geodesic paths.
\end{abstract}

\section{Introduction}

Consider a closed, hyperbolic surface $(V,\mathfrak{g})$. The fundamental group $\Gamma = \pi_1(V)$ acts cocompactly on the hyperbolic plane $(\mathbb{H}^2,d)$. Through this action, the conjugacy classes $\text{conj}(\Gamma)$ in $\Gamma$ are in bijection with the closed geodesics on $(V,\mathfrak{g})$. If we equip $\Gamma$ with a finite generating set $S$ then a closed geodesic $\gamma$ on $(V,\mathfrak{g})$ can be assigned two natural lengths: the geometric length $\ell(\gamma)$ (according to the Riemannian metric) and the word length $|\gamma|_S$ (the $S$ length of the shortest word in the corresponding conjugacy class). These lengths are comparable in the sense that there exists $C > 1$ such that
\begin{equation}\label{eq.comp}
C^{-1} |\gamma|_S \le \ell(\gamma) \le C |\gamma|_S
\end{equation}
for all closed geodesics $\gamma$. We are then led to ask if there is a more refined result that compares these lengths on average. That is, can we prove counting limit laws that compare these lengths? This and related questions have been studied extensively, see \cite{milnor}, \cite{horsham.sharp}, \cite{rivin}, \cite{CalegariFujiwara2010}, \cite{chas.lalley}, \cite{GTT2},  \cite{GTT}, \cite{cantrell.sert} amongst many other works.

In this work we study an analogue of the above question in the setting of translation surfaces. Translation surfaces are surfaces that are flat except at finitely many points called singularities. These singularities can be thought of as concentrated points of negative curvature. Given our above discussion about hyperbolic surfaces it is natural to ask whether we can compare analogues of word length and geodesic length for geodesic paths on translation surfaces.

\subsection{Translation surfaces}\label{sec.introts}
 A translation surface is a  compact Riemann surface  $X$ equipped with a  (non-trivial)  holomorphic one-form $\omega$.
Equivalently, $X$ is a compact surface with a flat metric except at a finite set $\mathcal{X}= \{x_1, \ldots, x_n\}$ of singular points with cone angles $2\pi (k(x_i)+1)$, where $k(x_i) \in \mathbb N$, for $i=1, \ldots, n$. The singularities correspond to the zeros of $\omega$. In this work we will always assume that $\mathcal{X}$ is non-empty, i.e. our surface has singularities. In a translation surface, a path which does not pass through singularities is a locally distance minimising geodesic if it is a straight line segment. This is also true for saddle connections: straight lines that start and end and singularities and do not pass through any other singularities. We will write $\mathcal{S}$ for the collection of saddle connections. 

There is a rich body of research that studies various aspects of translation surfaces. One particular focus has been to understand  the lengths of saddle connections. See for example \cite{masur}, \cite{eskin2003}, \cite{AC}, \cite{CD},  \cite{Dozier}, \cite{Pollicott-Colognese}, \cite{AFM}, \cite{CF}. 


There are infinitely many local geodesics emanating from a fixed base point on $(X,\omega)$. Many of these do not visit singularities and so correspond to straight lines. In this work, we restrict our study to local geodesics that visit singularities. That is, we study and compare various lengths associated to saddle connection paths. 
 \begin{definition}
A \textit{saddle connection path} $p
= (s_{i_1},\ldots,s_{i_n})$  is a finite string of oriented saddle collections $s_{i_1},\ldots,s_{i_n} \in \mathcal{S}$ which form a local geodesic path in $(X,\omega)$. Equivalently a saddle connection path is a local geodesic on $(X,\omega)$ that starts and ends at a singularity.
\end{definition}
\begin{figure}
\begin{tikzpicture}[scale=0.8, every node/.style={scale=0.8}]
\draw[gray, thick] (0,0)--(6,0);
\draw[gray, thick] (3,3)--(3,6);
\draw[gray, thick] (0,0)--(0,6);
\draw[gray, thick] (6,0)--(6,3);
\draw[gray, thick] (3,6)--(0,6);
\draw[gray, thick] (3,3)--(6,3);
\draw[gray, thick] (0,0)--(0,3);
\draw[gray, thick] (6,0)--(6,3);
\draw[gray, thick] (6,9)--(6,9);
\draw [fill=red] (0,0) circle [radius=0.08];
\draw [fill=red] (0,3) circle [radius=0.08];
\draw [fill=red] (3,3) circle [radius=0.08];
\draw [fill=red] (3,6) circle [radius=0.08];
\draw [fill=red] (0,6) circle [radius=0.08];
\draw [fill=red] (3,0) circle [radius=0.08];
\draw [fill=red] (6,0) circle [radius=0.08];
\draw [fill=red] (6,3) circle [radius=0.08];
\node at (1.5,-0.4) {$1$};
\node at (4.5,-0.4) {$2$};
\node at (1.5,6.4) {$1$};
\node at (4.5,3.4) {$2$};
\node at (-0.4,1.5) {$a$};
\node at (6.4,1.5) {$a$};
\node at (-0.4,4.5) {$b$};
\node at (3.4, 4.5) {$b$};
\end{tikzpicture}
\hskip 0.25cm
\begin{tikzpicture}[scale=0.8, every node/.style={scale=0.8}]
\draw[gray, thick] (0,0)--(6,0);
\draw[gray, thick] (3,3)--(3,6);
\draw[gray, thick] (0,0)--(0,6);
\draw[gray, thick] (6,0)--(6,3);
\draw[gray, thick] (3,6)--(0,6);
\draw[gray, thick] (3,3)--(6,3);
\draw[gray, thick] (0,0)--(0,3);
\draw[gray, thick] (6,0)--(6,3);
\draw[gray, thick] (6,9)--(6,9);
\draw[very thick, blue, -latex] (0,6)--(3,3);
\draw[very thick, blue, -latex] (3,3)--(6,3);
\draw[very thick, blue] (3,3)--(0,1.5);
\draw[very thick, blue, -latex] (6,1.5)--(3,0);
\draw [fill=red] (0,0) circle [radius=0.08];
\draw [fill=red] (0,3) circle [radius=0.08];
\draw [fill=red] (3,3) circle [radius=0.08];
\draw [fill=red] (3,6) circle [radius=0.08];
\draw [fill=red] (0,6) circle [radius=0.08];
\draw [fill=red] (3,0) circle [radius=0.08];
\draw [fill=red] (6,0) circle [radius=0.08];
\draw [fill=red] (6,3) circle [radius=0.08];
\node at (1.5,-0.4) {$1$};
\node at (4.5,-0.4) {$2$};
\node at (1.5,6.4) {$1$};
\node at (4.5,3.4) {$2$};
\node at (-0.4,1.5) {$a$};
\node at (6.4,1.5) {$a$};
\node at (-0.4,4.5) {$b$};
\node at (3.4, 4.5) {$b$};
\node at (1.75, 4.75) {$s_1$};
\node at (4.5, 2.7) {$s_2$};
\node at (1.62, 1.87) {$s_3$};
\end{tikzpicture}
\caption{ A simple way to present a translation surface is by identifying opposite sides of a polygon in the plane. On the left is an L-shaped translation surface with one singularity (indicated by the dots). The picture on the right shows a saddle connection path $p = (s_1,s_2,s_3)$ on the surface. The angle between $s_1$ and $s_2$ is $5\pi/4$ and the angle between $s_2$ and $s_3$ is $13\pi/6$.}
\end{figure}
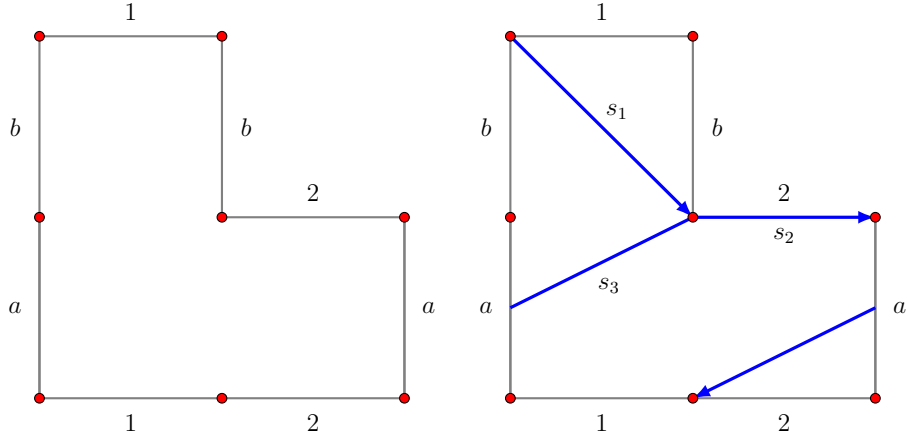

\noindent We write 
\begin{enumerate}[(i)]
\item $\ell(p) = \ell(s_1)+\ell(s_2)+\cdots +\ell(s_n)$ for the sum of the lengths of the constituent saddle connections; and 
\item $|p|$ for the number of singularities visited by $p$, i.e. if $p = (s_{i_1},\ldots,s_{i_n})$ then $|p| = n+1$.
\end{enumerate}
We refer to $\ell(p)$ as the \textit{geometric length} of $p$ and and we refer to $|p|$ as the \textit{singularity} length of $p$.
These quantities will play the role of the geometric length and word lengths associated to closed geodesics on surfaces that we discussed at the start of this work. We will present a statistical comparison between these quantities. 

Given a translation surface $(X,\omega)$, and $x \in \mathcal{X}$ a base point singularity, let $\P_x$ denote the collection of saddle connection paths $p$ starting at $x \in \mathcal{X}$. Write $\P_x(T)$ for the collection of saddle connection paths $p\in \P_x$ with  $\ell(p) <T$ and let $\mu_T$ denote the uniform counting measure on $\P_x(T) \subset \P_x$. 
\begin{theorem}\label{thm.sing}
Let $(X,\omega)$ be a translation surface and fix a singularity $x \in\mathcal{X}$. Let $|\cdot|$ and $\ell(\cdot)$ denote the singularity and geometric length of a saddle connection path respectively. Then there exists $\Lambda > 0$ such that for any $\epsilon > 0$ 
\[
\limsup_{T\to\infty} \frac{1}{T} \log \left( \mu_T \left\{ p \in \P_x : \left| \frac{|p|}{T} - \Lambda \right| > \epsilon \right\}\right) < 0.
\]
Furthermore there exists $\sigma^2 > 0$ such that for any $a,b \in \R, a<b$
\[
\lim_{T\to\infty} \mu_T\left\{ p\in \P_x: \frac{|p| - \Lambda T}{\sqrt{T}} \in [a,b] \right\} = \frac{1}{\sqrt{2 \pi} \sigma}\int_a^b e^{-t^2/2\sigma^2} \ dt.
\]
\end{theorem}
Unlike for the word and geometric length of closed geodesics on surfaces, $|p|$ and $\ell(p)$ are not comparable as in equation $(
\ref{eq.comp})$ : we can find a sequence $p_n$ of closed saddle paths for which the quotient $|p_n|/\ell(p_n)$ converges to $0$ as $n\to\infty$. Nonetheless, our theorem shows that on average, a saddle connection path $p$ visits approximately $\Lambda \ell(p)$ singularities and that (along with the appropriate normalisation) $|p|$ distributes according to a non-degenerate central limit theorem.

We will deduce Theorem \ref{thm.sing} from a  more general comparison result, Theorem \ref{thm.general} which compares geometric length with other quantities. Indeed, there are many natural quantities, other than singularity length and geometric length, which we can associate to a saddle connection path $p$. These include:
\begin{enumerate}
\item the total change in the real (or imaginary)  part when walking along $p$ ; and,
\item the total angle change gone through when traversing $p$.
\end{enumerate}
We now discuss these quantities in more detail.\\
(1) Using the one-form $\omega$ we have a well defined notion of real and imaginary part for a saddle connection: given a saddle connection $s$, the holonomy vector 
\[
\int_s \omega =: \mathfrak{R}_s + i \mathfrak{I}_s \in \C
\]  
encodes the real $\mathfrak{R}_s$ and imaginary $\mathfrak{I}_s$ change along $s$.  We define $|p|_\mathfrak{R}$ and $|p|_\mathfrak{I}$ to be the sum of these changes, respectively, along the saddle connection path $p$.  \\
(2) Suppose that $p$ passes through the singularity $x$. Then $p$ (possibly) changes direction at $x$. We can measure the angle changed by $p$ to get number in $[0,\pi k(x)]$ (where $k(x)$ is as above). This angle is the smallest change (between measuring both clockwise and anticlockwise) made by $p$. Summing these  angle changes along the singularities visited gives the total angle change along $p$, which we denote $|p|_\theta$.

Our general results apply to a large class of functions on saddle connection paths. Suppose that $\varphi : \mathcal{S} \to \R$ is a function that assigns to saddle connection $s$ a real number $\varphi(s)$. Assume further that there exists $C > 0$ such that  $|\varphi(s)| \le C \ell(s)$ for all $s \in \mathcal{S}$. Then we extend $\varphi$ to saddle connection paths in the obvious way: if $p = (s_{i_1}, \ldots, s_{i_n})$ then we set
\[
\varphi(p) = \varphi(s_{i_1}) + \cdots + \varphi(s_{i_n}).
\]
We call such functions \textit{saddle cost functions}.
The functions listed in $(1)$ and $(2)$ as well as the singularity length $|\cdot|$ can all be realised as saddle cost functions.
Our general result is as follows.
\begin{theorem}\label{thm.general}
Let $(X, \omega)$ be a translation surface and fix $x \in \mathcal{X}$. Let $\ell(\cdot)$ denote the  geometric length and let $\varphi(\cdot)$ be a saddle cost function. Then there exists $\Lambda \in \R$ such that for any $\epsilon > 0$ 
\[
\limsup_{T\to\infty} \frac{1}{T} \log \left( \mu_T \left\{ p \in \P_x : \left| \frac{\varphi(p)}{T} - \Lambda \right| > \epsilon \right\}\right) < 0.
\]
Furthermore there exists $\sigma^2 \ge 0$ such that for any $a,b \in \R, a<b$
\[
\lim_{T\to\infty} \mu_T\left\{ p\in \P_x: \frac{\varphi(p) - \Lambda T}{\sqrt{T}} \in [a,b] \right\} = \frac{1}{\sqrt{2 \pi} \sigma}\int_a^b e^{-t^2/2\sigma^2} \ dt.
\]
We also have the following:
\begin{enumerate}[(i)]
\item if $\varphi$ is positive (i.e. $\varphi(p) \ge 0$) and is not identically $0$ then $\Lambda > 0$;
\item we have that $\sigma^2 = 0$  if and only if there exists $\tau \in \R$ such that $\varphi(p) = \tau \ell(p)$  for all closed saddle connection paths $p$ (i.e. saddle connection paths that start and end at the same singularity); and,
\item $\Lambda$ and $\sigma^2$ are independent of the choice of $x \in \mathcal{X}$.
\end{enumerate}
\end{theorem}
In Section \ref{sec.variance} we provide a general condition that guarantees positive variance in Theorem \ref{thm.general}.
\begin{remark}
The mean $\Lambda$ and variance $\sigma^2$ in the above theorem can be obtained from the asymptotic expressions
\[
\Lambda = \lim_{T\to\infty} \frac{1}{\#\P_x(T)} \sum_{p \in \P_x(T)} \frac{\varphi(p)}{T} \ \ \text{ and } \ \ \sigma^2 = \lim_{T\to\infty} \frac{1}{\#\P_x(T)} \sum_{p \in \P_x(T)} \frac{(\varphi(p) -\Lambda T)^2}{T} \ge 0.
\]
\end{remark}
 We also obtain a counting large deviation theorems and non-degenerate central limit theorems comparing the geometric length with the quantities listed as $(1)$ and $(2)$ above, see Section \ref{sec.mainresults}. We also provide a general condition that guarantees that the variance of for the central limit theorem is strictly positive. 

Theorem \ref{thm.general} provides a precise statistical comparison between the geometric length and a saddle cost function.
It is natural to ask whether we can compare multiple saddle cost functions at a time so that we can study the distribution of the holonomy vector (and other examples). This leads us to a prove a multi-dimensional version of this theorem that compares the geometric length with a collection of saddle cost functions. 
\begin{theorem}\label{thm.generalmulti}
Suppose that $\overline{\varphi} : \P_x \to \mathbb{R}^k$ is a vector of $k$ saddle cost functions. Then, there exist $\overline{\Lambda} \in \R^k$ such that for any $\epsilon > 0$
\[
\limsup_{T\to\infty} \frac{1}{T} \log \left(\mu_T\left\{ p \in \P_x : \left\| \frac{\overline{\varphi}(p)}{T} - \overline{\Lambda} \right\| > \epsilon \right\} \right) < 0
\]
where $\| \cdot\|$ is any fixed norm on $\R^k$.
Furthermore, there is a positive definite matrix $\Sigma \in \GL_k(\R)$ such that the distribution of
\[
 \frac{\overline{\varphi}(p) - \overline{\Lambda} T}{\sqrt{T}} \ \ \text{ with respect to $\mu_T$ }
\]
converges as $T\to\infty$ to the multi-dimensional Gaussian distribution with mean $0$ and covariance matrix $\Sigma$.
Furthermore $\Sigma$ is strictly positive definite if and only if there does not exist non-zero  $t \in \R^k$ such that for all closed saddle connection paths $p$, 
\[
\langle t, \overline{\varphi}(p) - \overline{\Lambda} \ell(p)\rangle = 0.
\]
\end{theorem}

We now discuss applications of this result. Given a saddle connection path $p = (s_{i_1}, \ldots, s_{i_n})$ we write
\[
\int_p \omega = \sum_{i=1}^n \int_{s_i} \omega
\]
for the sum of the holonomy vectors along the saddle connections making up $p$. We think of $\int_p \omega$ as being in $\R^2$ through the natural identification of $\C$ with $\R^2$. We prove the following.
\begin{corollary}\label{cor.holonomy}
Let $(X,\omega)$ be a translation surface and fix a singularity $x \in\mathcal{X}$. Then for any $\epsilon > 0$ we have that
\[
\limsup_{T\to\infty} \frac{1}{T} \log \left( \mu_T\left\{ p \in \P_x : \left\| \frac{1}{T} \int_p \omega  \right\| > \epsilon \right\} \right) < 0
\]
where $\|\cdot\|$ is any norm on $\R^2$.
Furthermore, there exists a strictly positive definite symmetric matrix $\Sigma \in \GL_2(\R)$ such that, as $T\to\infty$, the distribution of the vectors
\[
\frac{1}{\sqrt{T}} \int_p \omega \in \R^2 \ \ \text{ with respect to $\mu_T$}
\]
converges to the 2-dimensional Gaussian distribution with mean $0$ and covariance matrix $\Sigma$.
\end{corollary}

As another application, take a subset $Y = \{ y_1,\ldots, y_k \} \subset \mathcal{X}$ of singularities of $X$. Let $|\cdot|_Y: \P_x \to \R^k$ denote the vector
\[
|p|_Y = (|p|_{y_1}, \ldots, |p|_{y_k})
\]
where each $|p|_{y_j}$ denotes the number of times $p$ visits the singularity $y_j$. Clearly each $|\cdot|_{y_j}$ is a saddle cost function and so we can apply Theorem \ref{thm.generalmulti} to obtain a multi-dimensional large deviation and central limit theorems. Furthermore the central limit theorem is non-degenerate.
\begin{corollary}\label{cor.sing}
Let $(X,\omega)$ be a translation surface and fix a singularity $x \in\mathcal{X}$. 
There exist a  strictly positive vector $\Lambda(Y) \in \R^k_{>0}$ and a strictly positive definite matrix $\Sigma \in \GL_k(\R)$ such that, as $T\to\infty$, the distribution of
\[
 \frac{|p|_Y - \Lambda(Y) T}{\sqrt{T}} \ \ \text{ with respect to $\mu_T$ }
\]
converges to the $k$-dimensional Gaussian distribution with mean $0$ and covariance matrix $\Sigma$.
\end{corollary}

We now discuss our approach to proving the above results.
\subsection{Method and ideas behind the proof}
We first establish counting limit laws for infinite graphs (that satisfy some natural assumptions which we label $(\text{G}1), (\text{G}2), (\text{G}3)$). That is, we compare length functions on infinite graphs and prove that they satisfy a counting central limit theorem and large deviation theorem.
This is done in Section \ref{sec.intrograph} and we now briefly discuss how we prove these results.

In previous works, counting limit laws have often been established using Hwang's Quasi-power Theorem \cite{hwang}: a result that translates knowledge about sequences of moment generating functions into counting limit laws. We briefly illustrate how this argument might work in the setting of Theorem \ref{thm.general}. Suppose that $(X,\omega)$ is a translation surface and $\varphi$ a saddle cost function as above. Hwang's Theorem states that if we have a so-called `quasi-power expression'
\begin{equation}\label{eq.mgf}
\sum_{\ell(p) <T} e^{s \varphi(p)} = Ce^{\sigma(s)} (1 + O(k_T)^{-1}) \ \text{ uniformly for $s$ in a complex neighbourhood of $0$}
\end{equation}
where $\sigma$ is analytic, and $k_T \to \infty$ as $T\to\infty$, then we have counting limit laws comparing $\varphi$ and $\ell$. Here $\sum_{\ell(p) <T} e^{s \varphi(p)}$ is the sequence of moment generating functions. 
 In fact, Hwang's theorem implies limit laws with precise error terms (depending on $\sigma$ and $k_T$). 

There are many works that deduce counting limit laws following this line of argument: see \cite{BV}, \cite{Jungwon1}, \cite{CV}, \cite{cantrell.pollicott}, \cite{Jungwon2}. To obtain the required quasi-power expression one can borrow techniques from analytic number theory: by considering the appropriate  $2$-variable complex function (a Zeta function or Poincar\'e series), one can express the sequence of moment generating functions as a sequence of contour integrals. One can then evaluate/estimate these integrals to obtain the quasi-power expression. To obtain the uniform error in the quasi-power expression, one needs to know that the corresponding complex function admits a uniform analytic extension, except for a pole, past its critical line of convergence in $\C$. To prove this uniform convergence one needs to have strong estimate on the operator norms for a certain family of linear operators called transfer operators. These estimates, usually referred to as Dolgopyat estimates are delicate and are only known to hold in a handful of settings. All of the works mentioned above rely on these estimates.

Unfortunately, these Dolgopyat estimates do not hold in our setting of infinite graphs or translation surfaces. Indeed, the Poincar\'e series
\[
\sum_{p \in \P_x(T)} e^{-s\ell(p)} \ \ \text{ for $s \in \C$}
\]
and its two variable analogue that encodes the moment generating functions in $(\ref{eq.mgf})$ has poles arbitrarily close to the critical line $\mathfrak{Re}(s) =h$ (where $h$ is the exponential growth rate of $\#\P_x(T)$). We therefore need to develop a Dolgopyat estimate free approach to obtain our counting limit theorems. Our proof is slightly different for the large deviations and the central limit theorem. For the large deviation theorem we show that it suffices to prove a weak version of the quasi-power expression. This result appears as Proposition \ref{prop.wqp} below. To prove our central limit theorem, Theorem \ref{thm.mclt}, we generalise an argument due to Hwang and Janson \cite{hwang.janson}. 
This generalisation is both to multidimensional settings and also to settings in which we do not have good estimates on transfer operators.

Our proofs also rely on a key observation of Hofbauer and Keller \cite{HK} that was used by the second author and Colognese in \cite{Pollicott-Colognese} which shows that the infinite graphs we consider can, in some sense, be approximated by finite subgraphs. This intuition is formalised in expression $(\ref{eq.key})$ below which shows that proving the invertiblity of a certain infinite matrix (which encodes the infinite graph $\G$) can be reduced to proving invertibility of a finite matrix. We also need to show that the variance of our central limit theorem for infinite graphs can be (in some sense) approximated by the variances of central limit theorems on subgraphs.

Once we have proven our results in Section \ref{sec.intrograph}, we use them to deduce our result for translation surfaces. To do so, we show that it is possible to translate the assumptions $(\text{G1})$, $(\text{G2})$ and $(\text{G3})$ for graphs into analogous conditions for saddle connection paths. We verify that these conditions hold and also verify the non-degeneracy criteria for the central limit theorem. These proofs appear in Section \ref{sec.translation} in which we also provide explicit examples and applications.

To summarise, the outline of our proof is as follows:
\begin{enumerate}
\item We start by studying infinite graphs and reduce the infinite graph case to finite subgraphs;
\item we then prove our counting limit laws for graphs without relying on Dolgopyat estimates; and,
\item lastly, we prove that the results from the previous step can be applied to translation surfaces.
\end{enumerate}

\subsection*{Acknowledgements}
The authors are grateful to Jon Chaika and Selim Ghazouani for helpful discussions and comments. We are also grateful to Peter M\"uller for assisting us in proving Proposition \ref{prop.irrat}. MP's research supported by ERC grant 833802-resonances and EPSRC grant EP/T001674/1.
\section{Infinite directed graphs}\label{sec.intrograph}
Let $\G$ be a finite directed graph in which the edges $e$ in $\G$ are labelled with lengths $\ell(e)$. We can run a Markovian random walk on $\G$ and ask, for a typical path (with respect to the stationary distribution) consisting of $n$ edges, how long is the $\ell$ length of the path? This is a classical problem are there are many beautiful works that consider this and related 
questions, see \cite{Bougerol2}, \cite{Bougerol1}, \cite{Page}, \cite{Guivarch}. In these works randomness is introduced through the stationary distribution for the Markovian process. One could also consider deterministic limit laws in which the randomness is replaced by counting.  That is, we can consider all paths $P_n$ in $\G$ consisting of $n$ edges and ask, for uniformly selected $p \in P_n$, what do we expect the value of $\ell(p)$ to be? This question is well understood in the finite graph setting but such results are much less developed in the infinite setting. Indeed, in the finite setting there is a well developed approach for these problems that employs thermodynamic formalism and symbolic dynamics. In particular transfer operator techniques can be employed. Similar counting limit laws have been proved for some well-understood infinite dynamical systems \cite{BV}, \cite{CV}. As we discussed at the end of the introduction, the proof of these results rely on strong Dolgopyat estimates for transfer operators. The estimates do not hold in the current setting of infinite graphs.\\

Suppose we have a directed graph $\mathcal{G}$ with countable vertex set $V$ and a countable edges set $\mathcal{E}$. Label each edge $e \in \E$ with a length $\ell(e)$. We  need the following assumptions.\\

\noindent \textbf{Assumptions:}
\begin{enumerate}
\item[(G1)] for all $\sigma > 0$ we have that $\sum_{e \in \mathcal{E}} e^{-\sigma \ell(e)} < \infty$;
\item[(G2)] there exists a constant $C >0$ such that for each $e,e' \in \mathcal{E}$ there is a path of length less than $C + \ell(e) + \ell(e')$ in $\mathcal{G}$ which starts with $e$ and ends with $e'$;
\item[(G3)] there does not exist $d > 0$ such that $\{ \ell(p) : \text{$p$ is a closed path in $\mathcal{G}$}\} \subset d\mathbb{N}$.
\end{enumerate}

\begin{remark}
Intuitively these assumptions allow us to, in some sense, approximate the infinite graph $\G$ by finite subgraphs.
\end{remark}

Here $\ell(p)$ is the natural extension of $\ell$ from edges to paths: the length of a path is the sum of the corresponding edge lengths. Write $\mathcal{P}_v$ for the set of all closed paths starting with a fixed vertex $v$ and let $\mathcal{P}_v(T)$ denote the collection of closed paths of length at most $T$ in $\mathcal{P}_v$, i.e. $p \in \P_v$ with  $\ell(p) <T$.

\begin{definition}
A \textit{cost} on the edges is a  non-zero function $c : \mathcal{E} \to \R$ such that there exists a constant $C >0$ such that $|c(e)| \le 
C \ell(e)$ for all $e \in \E$.
\end{definition}
We can extend these costs to paths in the obvious way: the cost of a path $p$ is the sum of the weightings along the edges of the path. We denote this length by $c(p)$. 

\begin{example}
(1) One could take the cost function $c : \E \to \R$ that assigns each edge $1$. Then the cost of a path is the number of edges in the path.\\
(2) Fix a vertex $v \in \G$.  We can define a cost function $c : \E \to \R$ that assigns an edge $1$ if the edge ends at vertex $v$ and $0$ otherwise. The cost of a path $p$ is then equal to the number of times that the path visits the $v$ vertex.\\
(3) In general we could take the cost function $c : \E \to \R$ that assigns positive lengths to each edge to compare two different length labelings on $\G$.
\end{example}

We would like to form a statistical comparison between the lengths the costs of paths in $\G$.
We prove the following statistical results. In the following fix $v \in V$ and let $\mu_T$ denote the uniform counting probability measure on $\P_v(T) \subset \P_v$: for $E \subset \P_v$ 
\[
\mu_T(E) = \frac{1}{\#\P_v(T)} \#\{ p\in E \cap \P_v(T)\}.
\]

\begin{theorem}\label{thm.average}
Suppose $\mathcal{G}$ is a directed graph satisfying  assumptions $(\textnormal{G}1) - (\textnormal{G}3)$. Let $c : \E \to \R$ be a cost on the edges of $\G$.  Then there exists a constant $\Lambda(c) \in \R $ such that
\[
\frac{1}{T} \int_{\P_v} c(p) \ d\mu_T = \frac{1}{\#\mathcal{P}_v(T)} \sum_{\ell(p) < T} \frac{c(p)}{T} \to  \Lambda(c)
\]
as $T\to\infty$. If $c$ is positive ($c(p) \ge 0$ for all $p$) and is not identically $0$ then $\Lambda(c) > 0$.
\end{theorem}

This shows that there is an asymptotic average cost across paths in $\mathcal{P}_v$. We can then ask for more refined statistical results, for example, large deviation and central limit theorems. We will prove the following higher dimensional statistical limit laws. In the following when we have $n$ cost functions $c_1, \ldots, c_n$  we will write $\overline{c} : \E \to \R^n_{\ge 0}$ for the vector of costs $(c_1(e), \ldots, c_n(e))$.
\begin{theorem}\label{thm.ldt}
Let $\mathcal{G}$ be a directed graph satisfying assumptions $(\textnormal{G}1) - (\textnormal{G}3)$. Suppose $c_1,\ldots, c_n$ are $n$ costs on the edges. Then there exists a vector $\Lambda(\overline{c}) \in \R^n_{> 0}$ such that for any $\epsilon >0$
\[
\lim_{T\to\infty} \frac{1}{T} \log \left( \mu_T\left\{ p \in \P_v : \left\|\frac{\overline{c}(p)}{T} - \Lambda(\overline{c}) \right\| > \epsilon \right\} \right) < 0
\]
as $T\to\infty$. Here $\|\cdot\|$ is any norm on $\R^n$.
\end{theorem}

We then have the following central limit theorem.

\begin{theorem}\label{thm.mclt}
Let $\mathcal{G}$ and $c_1,\ldots, c_n$ and $\Lambda(\overline{c})$ be as in Theorem \ref{thm.ldt}.  Then we have the following. There exists a positive definite, symmetric matrix $\Sigma \in \GL_d(\R)$ such that the distribution of
\[
 \frac{\overline{c}(p)- \Lambda(\overline{c}) T}{\sqrt{T}} \ \ \text{ with respect to $\mu_T$ }
\]
converges as $T\to\infty$ to the multi-dimensional Gaussian distribution with mean $0$ and covariance matrix $\Sigma$.
Furthermore, $\Sigma$ is strictly positive definite if and only if there does not exist  non-zero $t \in \R^n$ such that for all closed paths $p$, 
\[
\langle t, \overline{c}(p) - \Lambda(\overline{c})\ell(p)\rangle = 0.
\]
\end{theorem}

In this theorem the matrix $\Sigma$ has entries
\[
\Sigma_{i,j}  = \lim_{T\to\infty} \frac{1}{T}  \int_{\P_v} (c_i(p) - \Lambda_i T)( c_j(p) - \Lambda_j T) = \lim_{T\to\infty} \frac{1}{\#\P_v(T)} \sum_{\ell(p) < T} \frac{(c_i(p)-\Lambda_i T)(c_j(p) - \Lambda_j T)}{T}
\]
where $\Lambda_i, \Lambda_j$ are the averages for $c_i, c_j$ from Theorem \ref{thm.average}. We also show the following.

\begin{proposition}\label{prop.independent}
The quantities $\Lambda(\overline{c})$ and $\Sigma$ introduced in the above theorems do not depend on the choice of $v \in V$.
\end{proposition}

These result, for graphs, which we believe are of independent interest will help us to show our results for translation surfaces.


 \section{Counting for graphs} \label{sec.graph}
\subsection{Preliminaries}
In this section consider a 
graph   $\mathcal G = (V,\E)$  and length function $\ell$
which satisfy the hypotheses $(\text{G}1), (\text{G}2), (\text{G}3)$.
We assume that the edge set $\E = \{e_1, e_2, \ldots\}$ is ordered by non-decreasing length.
For $e \in \E$ we will write $i(e), t(e) \in V$ for the initial and terminal vertex of $e$ respectively. 
Let $c_1,\ldots, c_n $ be costs and write $\overline{c}$ for the vector of these costs. As in the previous section we use $\P_v$ and $\P_v(T)$ to denote the collection of paths starting with the vertex $v$ and the paths starting at the edge $v$ with length at most $T$.
\begin{definition}
 We can associate to $\mathcal G$ the matrix $M$ 
 defined by 
\[
M(e,e')=\begin{cases} 1 &\mbox{if }  t(e)=i(e'), \\
0 & \mbox{otherwise. }  \end{cases} 
\]
For each  $s\in \mathbb C$ and $t \in \C^d$ 	
we define the perturbed matrix $M_{s,t}$ by 
\[
M_{s,t}(e,e')=M(e,e')e^{-s\ell(e') - \langle t, \overline{c}(e') \rangle}
\]
 for $e,e'\in \mathcal{E}$.
\end{definition}

Let $P(n,e,e')$ denote the set of paths in $\G$ consisting of $n$ edges, 
starting with edge $e$ and ending with edge $e'$. For any $n\geq 1$, we can write the $(e,e')^{th}$ entry of the $n^{th}$ power of the matrix as:
\[
M_{s,t}^n(e,e') = e^{s\ell(e) + \langle t, \overline{c}(e)\rangle}\sum_{p\in P(n+1,e,e')} e^{-s\ell(p) - \langle t, \overline{c}(p) \rangle}
\]
which will be finite by assumption $(\text{G}1)$. Moving forward, to simplify notation, we will use the enumeration $\E = \{e_1, e_2, \ldots\}$ (where we order the edges in non-decreasing length) and write $M_{s,t}(i,j)$ for the entry $M_{s,t}(e_i,e_j)$.

Note that when $\mathfrak{Re}(s)>0$  we can find $\epsilon(s) >0$ such that for all $t \in \C^d$ with $|t| \le \epsilon(s)$,  we can interpret $M_{s,t}$ as a bounded operator on $\ell^\infty(\mathbb{C})$ where
\[
M_{s,t}(u)=\Big(\sum_{j=1}^\infty M_{s,t}(i,j)u_j\Big)_{i=1}^\infty.
\]
To proceed, we would like to understand the domain of meromorphicity  of the linear operator $(I- M_{s,t})^{-1}: l^\infty(\mathbb C) \to l^\infty(\mathbb C)$.
To do so, we follow \cite{Pollicott-Colognese} which in turn uses an idea by Hofbauer and Keller \cite{HK}, to show that the invertibility of the $M_{s,t}$ depends only on the determinant of an associated finite sub-matrix.\\

Fix  $\epsilon > 0$ and, for convenience,  assume also that
\[
0 < \epsilon < h:= \limsup_{T \to \infty} \frac{1}{T} \log \#\P_v(T).
\]
It is not hard to see that $0 < h < \infty$  (see Lemma 2.3 of \cite{Pollicott-Colognese}).

We can truncate the matrix $M_{s,t}$ to the $k \times k$ matrix  $A_{s,t} = (M_{s,t}(i,j))_{i,j=1}^k$.
We then write
\[
M_{s,t} = 
\left(
\begin{matrix}
A_{s,t}&B_{s,t}  \\ C_{s,t} &D_{s,t}
\end{matrix}
\right)
\]
where $D_{s,t} = \left(M_{s,t}(i+k,j+k) \right)_{i,j=1}^\infty$. 

Again, we can interpret $I-D_{s,t}$ as a bounded linear operator  on $l^\infty(\mathbb C)$ 
and write   $(I-D_{s,t})^{-1}=\sum_{m=0}^\infty D_{s,t}^m$ if the operator $D_{s,t}$ has norm $\|D_{s,t}\|<1$.
This is true when  $\mathfrak{Re}(s) \geq \epsilon$ and $|t| \le \delta(\epsilon)$ for some $\delta(\epsilon) > 0$. This is because, by the definition of a cost function, for $s$ with $\mathfrak{Re}(s) \geq  \epsilon$  and $t \in \C^n$ with $|t| \le \delta$ 
we have that
\[
-\Re(s) \ell(p) - \Re(\langle t, \overline{c}(p) \rangle) \ge (\epsilon - C \delta)\ell(p)
\]
for some $C >0$.
Therefore, assuming $\delta$ is sufficiently small so that $\epsilon - C\delta > \epsilon/2$,
\[
\|D_{s,t}\|
\leq   \sup_{n\in\mathbb{N}}\sum_{m=1}^\infty |D_{s,t}(n,m)| 
\leq   \sum_{m=1}^\infty  e^{-(\mathfrak{Re}(s) - C\delta)\ell(m+k)} \leq   \sum_{m=1}^\infty  e^{-\epsilon \ell(m+k)/2} <1
\]
for $k$ sufficiently large. 

We can then verify by formal matrix multiplication that  
\begin{equation} \label{eq.mm}
I - M_{s,t} = 
\left(
\begin{matrix}
I - A_{s,t} - B_{s,t} (I-D_{s,t})^{-1}C_{s,t}&-B_{s,t} (I-D_{s,t})^{-1}  \\ 0 &I 
\end{matrix}
\right)
\left(
\begin{matrix}
I &0  \\ -C_{s,t} & I - D_{s,t}
\end{matrix}
\right).
\end{equation}

\noindent
We define the $k \times k$ matrix  $W_{s,t}:= A_{s,t} +B_{s,t} (I-D_{s,t})^{-1}C_{s,t}$, where each entry is given by a convergent series.
By (3.2), whenever $\det(I-W_{s,t})\neq 0$ then  we see that $I-M_{s,t}$ is invertible,  with inverse
\begin{equation}\label{eq.key}
(I-M_{s,t})^{-1}=\left(
\begin{matrix}
I &0  \\ (I - D_{s,t})^{-1}C_{s,t} & (I - D_{s,t})^{-1}
\end{matrix}
\right)
\left(
\begin{matrix}
(I-W_{s,t})^{-1}&(I-W_{s,t})^{-1}B_{s,t} (I-D_{s,t})^{-1}  \\ 0 &I 
\end{matrix}
\right).
\end{equation}
This leads to the following result.

\begin{lemma}\label{operator}
For each $s \in \C$ with $\mathfrak{Re}(s) > 0$ there exists $\epsilon(s), \delta(s) > 0$ such that the operator $(I-M_{s,t})^{-1}$ has a bi-analytic extension to $\{ |s| \le \delta(s)\} \times \{ |t| \le \epsilon(s)\}$
except when $\det(I-W_{s,t}) = 0$.
\end{lemma}

\begin{proof}
This follows from the identity (\ref{eq.mm}) and Hartogs' Theorem \cite[Theorem 1.2.5]{Hartogs} .
\end{proof}

We now turn our attention to studying the matrices $W_{s,t}$.

\subsection{The matrices $W_{s,t}$}

We begin with the following lemma. Recall that
a non-negative $k\times k$ matrix $M$ is {\it irreducible} if for each $i,j$ with $1\leq i,j\leq k$ there exists a natural number $m$ (depending on $i,j$) such that $(M^m)_{i,j}>0$.

\begin{lemma}
Take $\epsilon >0$. Then there exists $\delta(\epsilon) > 0$ such that $W_{s,t}$ is an irreducible matrix for all $\Re(s) > \epsilon$ and all real $ t \in \R^n$ with $|t| \le \delta(\epsilon)$. Further for each real $s > 0$,  $W_{s,0}$ is a non-negative irreducible  matrix and $W_{s,0}$ has a simple maximal positive eigenvalue
$\lambda(s)=\rho(W_{s,0})$, which depends  analytically on $s$ and satisfies $\lambda'(s)<0.$
\end{lemma}
\begin{proof}
Note that if $k$ is sufficiently large $A_{s,t}$ is irreducible by assumption $(\text{G}2)$. All other parts of the lemma follow from the Perron-Frobenius Theorem and analytic perturbation theory \cite{kato}.
\end{proof}

We will write $\lambda(s,t)$ for the leading eigenvalue of $W_{s,t}$ for $(s,t)$ in a neighbourhood of $(1,0)$ which exists by the above lemma.

\begin{lemma}\label{lem.deriv}
Suppose that the vector of cost function $\overline{c} : \P_v \to \R_{\ge 0}^n$  has component cost functions that are positive and not identically zero.
We then have that $\lambda_s(1,0,\ldots,0)  < 0$ and $\lambda_{t_i}(1,0,\ldots, 0) > 0$ for each $i=1,\ldots, n$. 
\end{lemma}
\begin{proof}
This is a standard computation: for $(s,t)$ in a neighbourhood of $(1,0)$ the matrices $W_{s,t}$  are analytic perturbations of $W_{1,0}$. It follows from analytic perturbation theory and the Perron-Frobenius Theorem that there exist analytically varying left and right eigenvectors $u(s,t), v(s,t)$ such that 
\[
\lambda(s,t)u(s,t) = u(s,t) W_{s,t} \ \ \text{ and } \ \ \lambda(s,t)v(s,t) = W_{s,t} v(s,t)
\]
for all $(s,t)$ in a neighbourhood $U \times V \ni (1,0)$. We assume $u(s,t)v(s,t) =1$ for all $(s,t) \in U \times V$. Differentiating the above expression and rearranging shows that 
\[
\lambda_s(1,0) = u(1,0) W_{1,0}'v(1,0) \ \ \text{ where $W_{1,0}'(i,j) = \frac{\partial }{\partial s}\Big|_{(s,t) =(1,0)} W_{s,t}(i,j)$.}
\]
It follows that $\lambda_s(1,0) < 0$. The other expressions following similarly.
\end{proof}

\subsection{Complex functions}
We can now introduce a  complex function whose analytic properties will be useful in deriving our results.
\begin{definition}
Define the complex function 
\[
\eta_{\mathcal G}(s,t) =\sum_{p\in \P_v} e^{-s\ell(p) - \langle t, \overline{c}(p)\rangle}, \quad s \in \mathbb C, t \in \C^d
\]
 where 
$\P_v = \{ p = e_1, \ldots, e_n : n \geq 0, \, i(e_1)=v \}$ is the set of paths in $\mathcal G$ starting at $v \in V$. 
\end{definition}

We first observe that $\eta_{\mathcal G}(s,t)$ converges to a bi-analytic function  on a neighbourhood of $(s,0)$ for any $\mathfrak{Re}(s) > h$.

For  $\Re(s)>0$ and $|t|$ sufficiently small, we define:
\begin{enumerate}
\item[(a)]
$\textbf{w}(s,t)=  (\chi_{\mathcal E_v}(i)e^{-s\ell(e) - \langle t, \overline{c}(e) \rangle})_{i=1}^\infty \in \ell^1$ where $\chi_{\mathcal E_v}$ denotes the characteristic function of  the set $\mathcal E_v = \{e \in \mathcal E \hbox{ : } i(e)=v\}$ of edges   whose  initial vertex  is $v$; and
\item[(b)] 
 $\textbf{1} = (1)_{i=1}^\infty \in \ell^\infty$ is the vector all of whose entries are equal to $1$.
\end{enumerate} 
We can then formally   rewrite $\eta_{\mathcal G}(s,t)$ as
\begin{equation}
\eta_{\mathcal G}(s,t) 
=\sum_{p\in \P_v} e^{-s\ell(p) - \langle t, \overline{c}(p) \rangle} =
\textbf{w}(s,t)^T \Big(\sum_{n=0}^\infty  M_{s,t}^n\Big) \, \textbf{1} = 
\textbf{w}(s,t)^T \Big(1 - M_{s,t}\Big)^{-1} \,  \textbf{1}.
\end{equation}
Observe that for $\mathfrak{Re}(s) \ge  \epsilon$ we have 
$\textbf{w}(s,t) \in \ell^1$ for all $|t| \le \delta(\epsilon)$ (some constant depending on $\epsilon$).
In fact we can write 
\begin{equation}
\eta_{\mathcal G}(s,t)=\frac{\phi(s,t)}{\det(I-W_{s,t})}
\end{equation}
where $\phi(s,t)$ is bi-analytic on $\mathfrak{Re}(s)> \epsilon$, $|t| < \delta(\epsilon)$. It is not hard to see that $\phi(h,0)$ is a positive real number.

Choose $k$ large enough such that  $(I-D_{s,t})$ is invertible, on the half plane $\Re(s)\geq \epsilon$ where $\epsilon<h$.
The following was shown in \cite[Proposition 4.5]{Pollicott-Colognese}
\begin{proposition}
For each $s \neq h$ with $\Re(s) = h$ the matrix $W_{s,0}$ has spectral radius at most $1$ and does not have $1$ as an eigenvalue.
\end{proposition}
In particular, since for small $t$ we have that $W_{s,t}$ is an analytic perturbation of $W_{s,0}$  we deduce the following.
\begin{proposition}\label{prop.perturb}
For each $s_0 \neq h$ with $\Re(s) \ge h$ there exist $\epsilon, \delta > 0$ such that the matrices $W_{s,t}$ for $|s-s_0| \le \epsilon, |t| \le \delta$ have spectral radius at most $1$ and do not have $1$ as an eigenvalue.
\end{proposition}
Before we move on to the proof of our statistical limit laws we state the Tauberian Theorem  \cite[Theorem III]{delange}  that will be crucial for our proof.

\begin{proposition}[Delange Tauberian Theorem] \label{prop.tau}
For a monotone increasing function $\phi : \mathbb{R}_{>0} \to \mathbb{R}_{>0}$ we set
\[
f(s) = \int_0^{\infty} e^{-sT} \ d\phi(T).
\]
Suppose that there is $\delta > 0$ such that
\begin{enumerate}
\item $f(s)$ is analytic on $\{\mathfrak{Re}(s) \ge \delta\} \backslash \{\delta\}$; and,
\item  there are positive integers $n, k  \ge 1$, an open neighbourhood $U \ni \delta$, non-integer numbers $0 < \mu_1, \ldots, \mu_k < n$ and analytic maps $g,h, l_1, \ldots, l_k : U \to \mathbb{C}$ such that
\[
f(s) = \frac{g(s)}{(s-\delta)^n} + \sum_{j=1}^k \frac{l_j(s)}{(s-\delta)^{\mu_j}}+ h(s) \ \text{ for $s \in U$ and such that $g(\delta) > 0$}.
\]
\end{enumerate}
Then
\[
\phi(T) \sim \frac{g(\delta)}{(n-1)!} T^{n-1} e^{\delta T}
\]
as $T\to\infty$.
\end{proposition}

In this result and throughout the rest of this work, for two functions $f, g : \R \to \R_{>0}$ we write $f(T) \sim g(T)$ as $T\to\infty$ if $f(T)/g(T) \to 1$ as $T\to\infty$.

We are now ready to move on to the proofs of our main results. By scaling $\ell$ we can and will always operate under the following assumption: \\

\noindent \textbf{Assumption:} we will always assume that the lengths of the edges of $\G$ have been normalised so that the exponential growth rate of $\#\P_v(T)$ is $h =1$.

\subsection{Law of large numbers}

We are now ready to prove our weak law of large numbers for graphs.

\begin{proof}[proof of Theorem \ref{thm.average}]
It suffices to prove this result for a single cost function (i.e. the one dimensional case). We will also assume, without loss of generality, that $c$ is strictly positive (since we can add a constant multiple of $\ell(\cdot)$ to guarantee this).
We define the two variable series
\[
\eta_\G(s,t) = \sum_{ p \in \P_v } e^{-s\ell(p) + tc(p)}
\]
for $s,t\in \C$.
Recall that
\[
\eta_\G(s,t) = \frac{\phi(s,t)}{\det(I -W_{s,t})}
\]
where $W_{s,t}$ is as defined in the previous section.
It follows from Proposition \ref{prop.perturb} that for any $s_0 \neq 1$ with  $\mathfrak{Re}(s_0) \ge 1$, there exist $\epsilon(s_0), \delta(s_0) > 0$ such that $\eta_\G$ is bi-analytic in $|s-s_0|\le \delta(s_0)$ and $|t| \le \epsilon(s_0)$.  For $(s,t)$ close to $(1,0)$ we can write 
\begin{equation}\label{eq.double}
\eta_\G(s,t) = \frac{\phi(s,t)}{(1-\lambda(s,t))}
\end{equation}
where $\phi(1,0) \in \R_{>0}$ and $\phi(s,t)$ is bi-analytic.
When $t =0$ we deduce that 
\[
\eta_\G(s,0) = \frac{  - \lambda_s(1,0)^{-1} \phi(1,0)}{s-1} + R(s)
\]
for $\Re(s) \ge 1$ where $R$ is analytic on this domain.
 It follows from Proposition \ref{prop.tau} (since $\lambda_s(1,0) < 0$ by Lemma \ref{lem.deriv}) that
\[
\#\P_v(T) \sim  -\lambda_s(1,0)^{-1} \phi(1,0) \, e^{T}
\]
as $T \to \infty$.

We now differentiate $(\ref{eq.double})$ with respect to $t$ at $t= 0$. Doing this shows that
\[
\sum_{p \in \P_v} c(p) e^{-s\ell(p)} =  \frac{d}{dt}\Big|_{t=0}\eta_\G(s,t) = \frac{ - \lambda_t(s,0) \phi(s,t)}{(1-\lambda(s,0))^2} + L(s)
\]
where $L$ is analytic on $\mathfrak{R}(s)\ge 1$ apart from a possible simple pole at $s=1$. This shows that the above series has an order $2$ pole at $s = 1$ since the derivative  $\lambda_t(1,0)$ is non-zero by Lemma \ref{lem.deriv}.

Now we have that the derivative of $\eta_\G$ has a pole of residue $\lambda_s(1,0)^{-2} \lambda_t(1,0) \varphi(1,0)$ at $s=1$. Applying Proposition \ref{prop.tau} gives that
\[
\sum_{\ell(p) < T} c(p) \sim \lambda_s(1,0)^{-2} \lambda_t(1,0) \varphi(1,0) \, T \, e^T
\]
as $T \to \infty$ or equivalently
\[
\lim_{T\to\infty} \frac{1}{\#P_v(T)}  \sum_{\ell(p)<T} \frac{c(p)}{T} = - \frac{\lambda_t(1,0)}{\lambda_s(1,0)} =: \Lambda(c)
\]
 as required. By Lemma \ref{lem.deriv}, $\Lambda(c) >0$ which proves the furthermore statement.
 \end{proof}


\subsection{Reparameterising the matrices}

To prove our central limit theorem and large deviation theorem it is convenient to work with a reparameterised version of the matrices $W_{s,t}$.
Let $\Lambda(\overline{c})$ be the asymptotic average for $\overline{c}$ that exists by Theorem \ref{thm.average}.
We define
\[
\widetilde{W}_{s,t} = W_{s + \langle t, \Lambda(\overline{c}) \rangle, t}.
\]
For $(s,t)$ close to $(1,0)$, $\widetilde{W}_{s,t}$ is an analytic perturbation of $W_{1,0}$ and so has an bi-analytically varying simple maximal eigenvalue $\widetilde{\lambda}(s,t)$. 
We can analogously define $\widetilde{B}_{s,t}, \widetilde{C}_{s,t}$ and $\widetilde{D}_{s,t}$ and
as in the previous section we can write $\widetilde{W}_{s,t} = \widetilde{A}_{s,t} + \widetilde{B}_{s,t}(I - \widetilde{D}_{s,t})^{-1} \widetilde{C}_{s,t}$. Note that for $(s,t)$ close to $(1,0)$, $\widetilde{W}_{s,t}$ has leading eigenvalue $\widetilde{\lambda}(s,t) = \lambda(s+\langle t, \Lambda(\overline{c})\rangle , t)$.

A simple calculation shows that
\begin{equation}\label{eq.non-zero}
\widetilde{\lambda}_{t_j}(1,0,\ldots,0) = 0 \ \text{ for each $j=1,\ldots,n$.}
\end{equation}

We also have the following result on positivity of the second derivatives. 

\begin{lemma}\label{lem.var}
Suppose that we are working with a single cost function, i.e. $t \in\R$.
Assume that there does not exist $\tau >0$ such that all periodic paths $p$ on the states $\{1,\ldots,k\}$ satisfy $c(p) = \tau \ell(p)$. 
If $k$ is sufficiently large then $\lambda(\widetilde{A})(s,t)$ has  a leading eigenvalue $\widetilde{A}_{s,t}$ on a neighbourhood of $(1,0)$ and there exists a constant $C >0$ such that 
\[
\widetilde{\lambda}_{tt}(1,0) \ge C \, \lambda(\widetilde{A})_{tt}(1,0) > 0.
\]
\end{lemma}

\begin{proof}
Note that for all $k$ large,  $\widetilde{A}_{1,0}$ is irreducible by $(\text{G}2)$.  Further, if we take $k$ sufficiently large then there does not exist $\tau >0$ such that all periodic paths $p$ on the states $\{1,\ldots,k\}$ satisfy $c(p) = \tau \ell(p)$. 

A standard computation (see for example \cite{parry.tuncel}) shows that
\[
\widetilde{\lambda}_{tt}(1,0) = \lim_{n\to\infty} \frac{1}{n} u(1,0) (\widetilde{W}''_{1,0})^n v(1,0)
\]
where $u(s,t)$, $v(s,t)$ are the normalised left and right eigenvectors for $\widetilde{W}_{s,t}$ and  $\widetilde{W}''(1,0)$ has entries 
\[
\widetilde{W}''_{1,0}(i,j) = \frac{\partial^2 }{\partial t^2} \Bigg|_{(1,0)} \widetilde{W}_{s,t}(i,j).
\]
However we can write $\widetilde{W}_{s,t}^n = \widetilde{A}_{s,t}^n + M_n(s,t)$ where $M_n(s,t)$ is a sequence of analytically varying $k \times k$ matrices and $M_n(1,0)'' > 0$. Since $\widetilde{A}_{1,0}$ is irreducible it has leading real eigenvalue $\lambda(\widetilde{A})(1,0)$ and strictly positive left and right eigenvectors $u_{\widetilde{A}}(1), v_{\widetilde{A}}(1)$ we deduce that there exists $C> 0$ such that 
\[
\widetilde{\lambda}_{tt}(1,0) \ge C \, \lim_{n\to\infty} \frac{1}{n} u_{\widetilde{A}}(1)(\widetilde{A}_{1,0}'')^n v_{\widetilde{A}}(1).
\] 
Here $\widetilde{A}_{1,0}''$ is the component-wise second derivative with respect to $t$ (analogous to $\widetilde{W}_{1,0}''$ above).
Lastly, from the finite state Livsic Theorem \cite[Lemma 3.7]{ParryPollicott}, the right hand side of the above equation is $0$ if and only if for all loops $p$ in staying in the first $k$ symbols $c(p) = \tau \ell(p)$ for some $\tau >0$. Hence it is strictly positive by assumption and the proof is complete.
\end{proof}

We can upgrade this result to the multidimensional version.

\begin{lemma}\label{lem.secondderiv}
If $k$ is sufficiently large then the eigenvalues of $\widetilde{W}_{s,t}$ satisfying the following.
We have that $\widetilde{\lambda}_s(1,0,\ldots,0)  < 0$.  Lastly, if we define the matrix $\Sigma =(\sigma_{i,j})$ with entries 
\[
\sigma_{i,j} = -\frac{\widetilde{\lambda}_{t_it_j}(1,0,\ldots,0)}{\widetilde{\lambda}_s(1,0,\ldots,0)}
\]
 then $\Sigma$ is positive definite. The matrix $\Sigma$ is strictly positive definite if there does not exist non-zero  $v \in \R^n$ such that for all closed paths $p$, 
\[
\langle v, \overline{c}(p) - \Lambda(\overline{c})\ell(p)\rangle = 0.
\]
\end{lemma}
\begin{proof}
The first derivative being strictly negative is standard. For the latter part use that we need only check the condition for $v$ in the $n$-sphere in $\R^n$. By compactness of the $n$-sphere (and continuity of the map $v \mapsto \langle v, \overline{c}(p) - \Lambda(\overline{c})\ell(p)\rangle$), if there does not exist  non-zero $v \in \R^n$ such that for all closed paths $p$, 
\[
\langle v, \overline{c}(p) - \Lambda(\overline{c})\ell(p)\rangle = 0
\]
 then we can find $k$ large so that there does not exist non-zero $v \in \R^n$ such that
 \[
 \langle v, \overline{c}(p) - \Lambda(\overline{c})\ell(p)\rangle = 0 \ \text{ for all closed paths $p$ that remain in the states $\{1,\ldots, k\}$.}
 \]
 Since the value of $v^T \Sigma v$ is, up to scaling by $-\widetilde{\lambda}_s(1,0,\ldots,0)$, the second derivative with respect to $t$ and $(1,0)$ of 
 \[
\C^2 \ni (s,t) \mapsto \lambda(s+t \langle v, \Lambda(\overline{c} \rangle , t\langle v,\Lambda(\overline{c}) \rangle)
 \]
   we may apply Lemma \ref{lem.var} to deduce the result.
\end{proof}

Using these results and the reparameterisation above, we will operate under the following standing assumption in the proofs of the large deviation and central limit theorems.\\

\noindent \textbf{Assumption:} We will assume that the cost functions $\overline{c}$ have been normalised so that $\lambda_s(1,0) < 0, \lambda_{t_j}(1,0) = 0$ for all $j =1,\ldots, n$ and so that the conclusion on Lemma \ref{lem.secondderiv} holds for $\lambda(s,t)$.

\subsection{Large deviations}

We will deduce our large deviation theorem we use the following local G\"artner-Ellis type theorem.
\begin{lemma}\label{lem.localldp}
Let $Z_T$ be a one-parameter family of real random variables and $\mu_T$ be a one-parameter family of probability measures on a space $X$.  Suppose that
there exists $\eta>0$ with
\[
\lim_{T\to\infty} \frac{1}{T}\log \mathbb{E}_T(e^{tZ_T}) = c(t)
\]
for all $t \in [-\eta, \eta]$ where $\mathbb{E}_T$ represents the expectation with respect to $\mu_T$.
If $c$ is continuously differentiable and convex on $[-\eta,\eta]$ and $c^\prime(0)=0$ then for any $\epsilon >0$
\[
\limsup_{T\to\infty} \frac{1}{T}\log \mu_T(|Z_T|>T\epsilon) < 0.
\]

\end{lemma}

\begin{proof}
We note that for $\epsilon >0$ and for any $ 0 < t < \eta$ we have that 
\[
\mu_T(Z_T > T \epsilon) \le \mathbb{E}_T(e^{t(Z_T - T\epsilon)}) = e^{-t T \epsilon} \, \mathbb{E}_T(e^{tZ_T}).
\]
Therefore we have that
\[
\limsup_{T\to\infty} \frac{1}{T} \log \mu_T(Z_T > T\epsilon) \le - t \epsilon + c(t).
\]
Since $c'(0) = 0$ and $c$ is convex, we can find $t_0 \in (0,\eta)$ such that $c(t_0) - t_0 \epsilon < 0$ and so we deduce that
\[
\limsup_{T\to\infty} \frac{1}{T} \log \mu_T(Z_T > T\epsilon) \le - t_0\epsilon + c(t_0) < 0.
\]
We can apply the same reasoning on the interval $[-\eta, 0)$ to get the exponential decay when $Z_T < T\epsilon$. The result follows.
\end{proof}

We also need the following asymptotic result. This can be seen as a weak version of the hypotheses of the Hwang Quasi-power Theorem mentioned in the introduction.

\begin{proposition}\label{prop.wqp}
There exist and open real neighbourhood $U \ni 0$, a real analytic function $\sigma: U \to \R_{> 0}$ such that the following holds. For any $t_0 \in U$ there is a positive constant $C_{t_0}$  such that
\[
\sum_{\ell(p)<T} e^{t_0 c(p)}  \sim C_{t_0} e^{\sigma(t_0) T}
\]
as $T\to\infty$. Furthermore $\sigma$ is convex on $U$.
\end{proposition}

\begin{proof}
We begin by applying the Implicit Function Theorem to find a real, open neighbourhood $U$ of $t=0$ and a real analytic function $\sigma : U \to \R$ such that $\lambda(\sigma(t),t) =1$ for all $t \in U$.

The Implicit Function Theorem also implies that
\[
\sigma'(0) =  - \frac{\lambda_t(1,0)}{\lambda_s(1,0)} = 0
\]
and 
\[
\sigma''(0)= - \frac{\lambda_{t t}(1,0)}{\lambda_s(1,0)} \ge 0
\]
by Lemma \ref{lem.deriv}.
In particular,  $\sigma$ is convex.
Then for any fixed $t_0\in U$ we can find $s_0 = \sigma(t_0)$ and a neighbourhood of $s_0$ such that
\begin{equation} \label{eq.g}
\eta_\G(s,t_0) = \frac{\phi(s,t_0)}{1-\lambda(s,t_0)}
\end{equation}
where $\phi$ is bi-analytic.
In particular, $\eta_\G(s,t_0)$ has a simple pole $s = s_0$ and we can check, that it has positive residue as in the proof of Theorem \ref{thm.average}. Furthermore, expression (\ref{eq.g})
shows us that $\eta_\G(s,t_0)$ does not have any other poles in the line $\Re(s) = s_0$. To see this note that the matrices $W_{s,t_0}$ can not have $1$ as an eigenvalue when $\Re(s) = s_0$ but $s \neq s_0$. This follows from precisely the same proof as the case $t_0 = 0$. Indeed if $1$ were an eigenvalue then this would imply that $\ell$ takes values in a lattice which contradicts hypothesis $(\text{G}3)$ (see the proof of \cite[Proposition 4.5]{Pollicott-Colognese}).

We also see that
\[
\eta_\G(s,t_0) =  \frac{\lambda_s(\sigma(t_0),t_0)^{-1} \phi(s,t_0)}{\sigma(t_0) - s} .
\]
Note that $\lambda_s(\sigma(t_0),t_0) > 0$ by the same argument as in Lemma \ref{lem.deriv}.
 To summaries we have shown that $\eta_\G(s,t_0) $ is analytic in a neighbourhood of $\Re(s) \ge s_0$ apart from a simple pole with positive residue at $s = s_0$.
 
Then we see by Proposition \ref{prop.tau} that there exists a positive constant $C_{t_0}$ depending on $t_0$ such that
\[
\sum_{\ell(p)< T} e^{t_0 c(p)}  \sim C_{t_0} e^{\sigma(t_0) T}
\]
as $T \to\infty$. 
\end{proof}

We are now ready to prove our large deviation theorem.

\begin{proof}[proof of Theorem \ref{thm.ldt}]
It suffices to prove the theorem when $n=1$.
Combining the above results shows that there is a neighbourhood  $U$ of $0$ and an analytic function $\sigma : U \to\R$  such that if $t \in U$ then 
\[
\lim_{T \to\infty} \frac{1}{T} \log \left( \frac{1}{\#\P_v(T)} \sum_{p \in \P_v(T) } e^{t c(p)} \right) = \sigma(t) - \sigma(0).
\] 
We can then apply Lemma \ref{lem.localldp} (note that $\sigma''(0) \ge 0$) to deduce that
\[
\limsup_{T \to\infty} \frac{1}{T} \log \left(  \mu_T\left\{ p \in \P_v, \, \left| \frac{c(p)}{T} - \Lambda(c) \right| > \epsilon \right\} \right) < 0.
\]
Here the random variables are $Z_T(p) = c(p)$ if $\ell(p) < T$ and $Z_T(p) = 0$ otherwise and the measures $\mu_T$ are the uniform counting measures on $\P_v(T)$. This concludes the proof.
\end{proof}
\begin{remark}
It is possible to improve Theorem \ref{thm.ldt} to a local large deviation principle, i.e. we can quantify the precise exponential decay rate in the large deviation theorem, at least for small $\epsilon > 0$. 
\end{remark}

\subsection{Central limit theorems}

Before moving on to the main result of this section, we establish some notation.
Given $\widehat{q} =(q_1,\ldots,q_n) \in \N^n$ we will write $|\widehat{q}| = q_1 +\cdots+q_n$. 
We will use $\varphi_1, \ldots, \varphi_n$ to denote the re-scaled cost functions $\varphi_j(p) = c_j(p) - \lambda(c_j) \ell(p)$.

The next proposition is based on a careful study of the following series
\[
\eta_{\w{q}}(s) = \sum_{p \in \mathcal{P}_v} \varphi_1(p)^{q_1}\cdots \, \varphi_n(p)^{q_n} e^{-s\ell(p)}.
\]
We will write
\[
\sigma_{i,j} = -\frac{\lambda_{t_it_j}(1,0,\ldots,0)}{\lambda_s(1,0,\ldots,0)} \ \ \text{ for $i,j = 1,\ldots,n$}
\]
as in Lemma \ref{lem.secondderiv}.
\begin{proposition}\label{prop.analyticity}
Given $\w{q} \in \mathbb{N}^n$ the function $\eta_{\widehat{p}}(s)$ is analytic in the plane $\mathfrak{R}(s) > 1$. Furthermore, $\eta_{\w{q}}(s)$ is analytic on $\mathfrak{R}(s) \ge 1$ apart from at $s = 1$. Moreover,  the nature  of the singularity at  $1$ depends on the  parity of $q = |\w{q}|$ as follows:\\
\noindent {\rm Case 1: $q$ is odd.} Then $\eta_{\w{q}}(s)$ has a possible finite, integer order poles at $s =1$ and is analytic otherwise. These poles have order at most $(q+1)/2$.\\
\noindent {\rm Case 2: $q$ is even}. In this case, there exists a positive definite, symmetric matrix $ \Sigma = (\sigma_{i,j})_{i,j=1}^n$
such that the following holds. For $s$ in a neighbourhood of $1$ we can write
\[
\eta_{\w{q}}(s) = \frac{R_{\w{q}}(s)}{(s-1)^{1+ \frac{q}{2}}}
\]
where each $R_{\w{q}}(s)$ is analytic and
\[
R_{\w{q}}(1) = C\left(\frac{q}{2}\right)! \sum_{i_1, \ldots, i_{q} } \sigma_{l(i_1),l(i_2)} \sigma_{l(i_3),l(i_4)} \cdots \, \sigma_{l(i_{q-1}), l(i_{q})}
\]
where:
\begin{enumerate}
\item
the sum over $i_1, \ldots, i_{q}$ is over the partition of the numbers $1,\ldots, q$ into disjoint pairs labelled $(i_1,i_2), \ldots, (i_{q-1},i_{q})$; and,
\item $l : \{1, \ldots, q \} \to \{1, \ldots,q\}$ sends the set $\{1,\ldots,q_1\}$ to $1$,  the set $\{q_1+1, \ldots, q_1+q_2\}$ to $2$ and continues in this way until finally sending $\{q_1+\cdots + q_{n-1} + 1, \ldots, q\}$ to $n$.
\end{enumerate}
\end{proposition}

\begin{proof}[Proof of Proposition \ref{prop.analyticity}]
We first show that $\eta_{\w{q}}$ is analytic at $\mathfrak{Re}(s) \ge 1$ other than $s=1$. To do so we recall that 
\[
\eta_\G(s,t) = \frac{\phi(s,t)}{\det(I -W_{s,t})}.
\]
Using Proposition \ref{prop.perturb} we see that $\eta_\G(s,t)$ has the right domain of analyticity and we can differentiate to get what we need. Studying the pole is the harder part. For $(s,t)$ close to $(1,0,0,\ldots,0)$ we can write 
\[
\eta_\G(s,t) = \frac{ F(s,t)}{1-\lambda(s,t)} + R(s,t)
\]
where $F(s,t), R(s,t)$ are multi-analytic in a neighbourhood of $(1,0)$. It is not hard to see that $F(1,0) \in \R_{>0}$. 
We now  want to study the  partial derivatives
\[
\left. \frac{\partial^k}{\partial t_1^{k_1} \partial t_2^{k_2} \cdots \, \partial t_n^{k_n}} \right|_{(s,0)}  \frac{F(s,t)}{1-\lambda(s,t)}
\]
for $\w{k} = (k_1, k_2,\ldots,k_n) \in \N^n$ with $|\w{k}| = k_1 + k_2 + \cdots + k_n =k$.
Since we are only interested in the largest order poles, it suffices to study
\[
\left. \frac{\partial^k}{\partial t_1^{k_1} \partial t_2^{k_2} \cdots \, \partial t_n^{k_n}} \right|_{(s,0)}  \frac{1}{1-\lambda(s,t)}  \text{ where $k_1+k_2 + \cdots + k_n= k$.}
\]
Using Fa\`a di Bruno's formula \cite{Faa} we have that these derivatives are given by
\[
\sum_{\pi \in \Pi_k} \frac{|\pi|!}{(1-\lambda(s,0))^{|\pi|+1}} \prod_{B \in \pi} \lambda_B(s,0)
\]
where
\begin{enumerate}
\item $\Pi_k$ is the set of partitions of $\{1,\ldots,k\}$;
\item $|\pi|$ is the number of blocks in $\pi$;
\item in the product, $B$ runs over the blocks in $\pi$; and,
\item $\lambda_B(s,0)$ is the partial derivative of $\lambda$ over the block $B$: if $B = \{ b_1,\ldots, b_l\}\subset \{ 1, \ldots, k\} $ then 
\[
\lambda_B(s,0) = \frac{\partial^l \lambda}{\partial t_1^{l_1} \partial t_2^{l_2} \cdots  \, \partial t_n^{l_n}} (s,0) , \text{ where }
\]
\[
l_1 = \# (B \cap \{ 1,\ldots, k_1\}),
l_2= \# (B \cap \{ k_1 + 1, \ldots, k_1+k_2\}), \ldots, 
l_n = \#(B \cap \{k_1 + \cdots + k_{n-1} +1, \ldots, k \}.
\]
\end{enumerate}
The key point to notice is that, by our assumption that  $\lambda_{t_j}(1,0) = 0$ for each $j =1,\ldots, n$, for any block $B$ of length $1$, $\lambda_B(1,0) = 0$. Hence, the Fa\`a di Bruni formula shows that when we are searching for the poles of highest order we can ignore terms coming from partitions that contain blocks of a single number. It therefore follows that when $k$ is odd, the highest order pole coming from the derivative above has order at most $(k - 3)/2 + 2  = (k+1)/2$. When $k$ is even the highest order poles come from partitions of $\{ 1,\ldots, k\}$ into pairs. We deduce that in this case the highest order poles are of the form
\[
\frac{(k/2)!}{(1-\lambda(s,0))^{k/2 + 1}} \sum_{\pi \in \Pi_k(2)} \prod_{B \in \pi} \lambda_B(s,0)
\]
where $\Pi_k(2)$ represents all partitions of $\{1,\ldots, k\}$ into blocks of size $2$. Using Lemma \ref{lem.secondderiv} we can then write
\[
\frac{(k/2)!}{(1-\lambda(s,0))^{k/2}} \sum_{\pi \in \Pi_k(2)} \prod_{B \in \pi} \lambda_B(s,0) = \frac{(-\lambda_s(1,0))^{k/2}(k/2)!}{(1-\lambda(s,0))^{k/2 +1}} \sum_{i_1,\ldots, i_k} \sigma_{l(i_1),l(i_2)}\cdots \sigma_{l(i_{k-1}),l(i_{k})} + g(s)
\]
where $g(s)$ has poles of integer orders strictly less that $k/2 + 1$ and $l : \{1, \ldots, k \} \to \{1,\ldots,k\}$ sends the set $\{1,\ldots,k_1\}$ to $1$,  the set $\{k_1+1, \ldots, k_1+k_2\}$ to $2$ and continues in this way until finally sending $\{k_1+\cdots + k_{n-1} + 1, \ldots, k\}$ to $n$ (as in the statement of the proposition).

To conclude the proof we sum over $0 \le k_1 \le q_1, 0 \le k_2 \le q_2$. To see that when $|\w{q}|$ is odd, $\eta_{\w{q}}$ has the required pole structure and when $|\w{q}| = q$ is even
\[
\eta_{\w{q}}(s) =  \frac{F(s,0)(-\lambda_s(1,0))^{q/2}(q/2)!}{(1-\lambda(s,0))^{q/2 +1}} \sum_{i_1,\ldots, i_p} \sigma_{l(i_1),l(i_2)}\cdots \sigma_{l(i_{q-1}),l(i_{q})} + G(s)
\]
where $G$ is analytic other than integer order poles of order at most $q/2$. To conclude we note that
\[
\frac{F(1,0)(-\lambda_s(1,0))^{q/2}(q/2)!}{(1-\lambda(s,0))^{q/2 +1}} = \frac{F(1,0)(q/2)!}{(-\lambda_s(1,0)) (s-1)^{q/2 + 1}} + H(s)
\]
where $H(s)$ is analytic on $\mathfrak{R}(s) \ge 1$ except for finite integer order poles at $s=1$ of order at most $q/2$.
This concludes the proof with
\[
C = - \frac{F(1,0)}{\lambda_s(1,0)} > 0 \ \text{ and the } \ \sigma_{i,j}\ \text{ as defined above.}
\]
\end{proof}

\subsection{Deducing the Central Limit Theorem}
In this section we will employ the estimates on the complex function described in the previous section 
to deduce the central limit theorem.  The approach is to use the method of moments, following a strategy inspired by Morris \cite{morris}.  

\begin{definition}
For each pair $\w{q} =(q_1,\ldots, q_n) \in \N^n$ we define
\[
\pi_{\w{q}}(T) = \sum_{p \in \P_v(T)} \varphi_{\w{q}}(p) \ \text{ where } \varphi_{\w{q}}(p) = \varphi_1^{q_1}(p)\cdots \, \varphi_n^{q_n}(p).
\]
\end{definition}

We can now use Proposition \ref{prop.tau} in the proof of the following moment estimate.

\begin{proposition}
When $|\w{q}| = q$ is even we have that
\[
\lim_{T\to\infty} \frac{1}{\#\P_v(T)} \sum_{p \in \P_v(T)} \left(\frac{\varphi_{\w{q}}(p)}{\sqrt{T}}\right)^q = \sum_{i_1, \ldots, i_{p} } \sigma_{\pi(i_1),\pi(i_2)} \sigma_{\pi(i_3),\pi(i_4)} \cdots \sigma_{\pi(i_{q-1}), \pi(i_{q})}.
\]
\end{proposition}

\begin{proof}
When all of the $q_j$ are even we can apply Proposition \ref{prop.tau} immediately to deduce the result.
When some of the $q_j$ are odd we have to work harder. There are a further $2$ sub-cases. \\

\noindent 
\textit{Sub-case $1$: $q/2$ is even.} When this is the case we define
\begin{align*}
G_1(s) &= \sum_{p \in \P_v} \left(\varphi_{2\w{q}}(p) + \ell(p)^q\right) e^{-s\ell(p)}\\
G_2(s) &= \sum_{p \in \P_v} \left(\varphi_{\w{q}}(p)+ \ell(p)^{q/2}\right)^2 e^{-s\ell(p)}\\
G_3(s) &= \sum_{p \in \P_v} \ell(p)^{q/2} \varphi_{\w{q}}(p) e^{-s\ell(p)}.
\end{align*}
We note that $G_2 = G_1 + 2G_3$.
Using Proposition \ref{prop.tau} in combination with Proposition \ref{prop.analyticity} we see that
\[
\frac{1}{\#\P_v(T)}\sum_{\ell(p) < T} \varphi_{2\w{q}}(p) + \ell(p)^q \sim \frac{T^q\left( R_{2\w{q}}(1) + R_{0,0}(1)(q!/(q/2)!)\right)}{Cq!}
\]
as $T\to\infty$. We also have that (since $q/2$ is even)
\[
\eta_{\w{q}}^{(q/2)}(s) = \sum_{p \in \P_v} \ell(p)^{q/2} \varphi_{\w{q}}(p)e^{-s\ell(p)} = G_3(p).
\]
Then using that $G_2 = G_1 + 2G_3$ we see that
\[
G_2(s) = \frac{R_{2\w{q}}(1) + R_{0,0}(1)(q!/(q/2)!) + 2 R_{\w{q}}(1)(q!/(q/2)!)}{(s-\lambda)^{1+q}} + f(s)
\]
where $f(s)$ is analytic, other than integer poles of order at most $q$. Therefore
\[
\frac{1}{\#\P_v(T)} \sum_{p \in \P_v(T)} \left(\varphi_{\w{q}}(p)+ \ell(p)^{q/2}\right)^2 
\]
grows asymptotically like
\[
\frac{R_{2\w{q}}(1) + R_{0,0}(1)(q!/(q/2)!) + 2 R_{q_1,q_2}(1)(q!/(q/2)!)}{C q!} \, T^q
\]
as $T\to\infty$. This implies that
\[
\frac{1}{\#\P_v(T)}\sum_{p \in \P_v(T)} \ell(p)^{q/2} \varphi_{\w{q}}(p)\sim \frac{R_{\w{q}}(1) T^{p}}{C(q/2)!} \ \text{ as $T\to\infty$.}
\]
We want the same expression but with $\ell(p)^{q/2}$ replaced by $T^{q/2}$. We now remove this weighting term. Note that
\[
\sum_{p \in \P_v(T)} \ell(p)^{q/2} \varphi_{\w{q}}(p) = \int_0^T t^{q/2} \ d\pi_{\w{q}}(t) = T^{q/2}\pi_{\w{q}}(T) - \frac{q}{2} \int_0^T t^{q/2 -1} \pi_{\w{q}}(t) \ dt.	
\]
Without loss of generality we can assume that $q_1,\ldots,q_{2k}$ are odd and $q_{2k}+1, \ldots, q_n$ are even for some $k \le \lfloor q/2 \rfloor$.

We now use the elementary inequality for real numbers $x_1,\ldots,x_{2k}$,
\[
\sum_{B \subset \{1,\ldots,2k\}, |B| = k} x_B^2 \ge |x_1 \cdots x_{2k}|
\]
where for $ B =\{b_1,\ldots, b_k\}  \subset \{1,\ldots,2k\}$ we set $x_B = x_{b_1}\cdots \, x_{b_k}$.

Using this we see that
\[
|\pi_{\w{q}}(t)| \le \sum_{B \subset \{1,\ldots,2k\}, |B| = k} \pi_{B(\w{q})}(t) 
\]
where $B(\w{q}) \in \N^n$ is the vector with entries $q_j$ for $j \nin \{1,\ldots,2k\}$, $q_j + 1$ if $j \in B \cap \{1,\ldots,2k\}$ and $q_j-1$ if $j \in  \{1,\ldots,2k\} \backslash B$.
Since all of the entries in $B(\w{q})$ are even and $|B(\w{q})| = |\w{q}| = q$ we can apply Proposition \ref{prop.tau} to deduce that
\[
\sum_{B \subset \{1,\ldots,2k\}, |B| = k} \pi_{B(\w{q})}(t)  = O(t^{q/2}e^t)
\]
as $t\to\infty$. It therefore follows that
\[
\int_0^T t^{q/2 -1} \pi_{q_1,q_2}(t) \ dt= o\left( T^{q/2} e^T\right)
\]
as $T\to\infty$.
This implies the required asymptotic in this sub-case.\\

\noindent
\textit{Sub-case $2$: $q/2$ is odd.} In this case we define
\begin{align*}
H_1(s) &= \sum_{p \in \P_v} \left(\varphi_{2\w{q}}(p)+ \ell(p)^q\right) e^{-s\ell(p)}\\
H_2(s) &= \sum_{p \in \P_v} \left(\varphi_{\w{q}}(p) + \ell(p)^{q/2}\right)^2 e^{-s\ell(p)}\\
H_3(s) &= \sum_{p \in \P_v} -\ell(p)^{q/2} \varphi_{\w{q}}(p) e^{-s\ell(p)}
\end{align*}
This time we have that $H_2 = H_1 - 2 H_3$. Following the same argument as before, we deduce that
\[
\frac{1}{\#\P_v(T)}\sum_{p \in \P_v(T)} \ell(p)^{q/2} \varphi_{\w{q}}(p)\sim \frac{R_{q_1,q_2}(1) T^{p}}{C(q/2)!} 
\]
as $T\to\infty$. Again, we can use the Stiltjes integral as before to change this expression into the required asymptotic expression.

\end{proof}

We now handle the odd sum case.

\begin{proposition}
When $q = |\w{q}|$ is odd we have that
\[
\lim_{T\to\infty} \frac{1}{\#\P_v(T)} \sum_{p \in \P_v(T)} \left(\frac{\varphi_{\w{q}}(p)}{\sqrt{T}}\right)^q = 0.
\]
\end{proposition}

\begin{proof}
As before we define
\begin{align*}
K_1(s) &= \sum_{p \in \P_v} \left(\varphi_{2\w{q}}(p) + \ell(p)^q\right) e^{-s\ell(p)}\\
K_2(s) &= \sum_{p \in \P_v} \left(\varphi_{\w{q}}(p)+ \ell(p)^{q/2}\right)^2 e^{-s\ell(p)}\\
K_3(s) &= \sum_{p \in \P_v} \ell(p)^{q/2} \varphi_{\w{q}}(p) e^{-s\ell(p)}
\end{align*}
and note that $K_2 = K_1 + 2K_3$ and 
\[
K_3(s) = \eta_{2\w{q}}(s) - \eta_0^{(p)}(s) = \frac{g(s)}{(s-1)^{1+q}} + f(s)
\]
where $g(1) > 0$ and $f,g$ are analytic.
By Proposition \ref{prop.tau} we deduce that
\[
\frac{1}{\#\P_v(T)} \sum_{\ell(p)<T} \varphi_{2\w{q}}(p)+ \ell(p)^q \sim \frac{g(1) T^q}{q!}
\]
as $T\to\infty$. We now calculate
\[
\eta_{\w{q}}^{\left( \frac{p+1}{2}\right)}(s) = (-1)^{\left(\frac{p+1}{2}\right)} \sum_{\ell(p) < T} \ell(p)^{\left( \frac{p+1}{2}\right)} \varphi_{\w{q}}(p) e^{-s\ell(p)}.
\]
Now using the identity
\[
\int_0^\infty t^{-1/2} e^{-tx} \ dt = \sqrt{\pi} x^{-1/2}
\]
for $x > 0$ we see that
\[
K_3(s) = \frac{(-1)^{\frac{p+1}{2}}}{\sqrt{\pi}} \int_0^\infty \frac{\eta_{\w{q}}^{\left( \frac{p+1}{2}\right)}(s+t)}{\sqrt{t}} \ dt
\]
and hence $K_3$ is analytic except for a pole of order at most $p+1$ at $s=1$. Now, we can write
\[
\eta_{\w{q}}^{\left( \frac{p+1}{2}\right)}(s) = \sum_{j=1}^{p+1} \frac{a_j}{(s-1)^j} + h(s)
\]
where $a_j \in \C$, $h$ is analytic in $\mathfrak{R}(s) \ge 1$. Then using the identity
\[
\int_0^\infty \frac{1}{(s+t - 1)^j \sqrt{t}} \ dt = \frac{\pi (2j -2)!}{2^{2j-1}(j-1)!^2} \frac{1}{(s-1)^{k - \frac{1}{2}}}
\]
(which follows by integration by parts) we deduce that
\[
K_3(s) = \sum_{j=0}^{p+1} \frac{c_j}{(s-1)^{j-\frac{1}{2}}} + l(s)
\]
where $c_j \in \C$ and $l$ is analytic in the half plane. Then using that $k_2 = K_1 + 2K_2$ we deduce that
\[
K_2(s) = \sum_{j=0}^{p+1} \frac{a_j}{(s-1)^j} + \frac{b_j}{(s-1)^{j-\frac{1}{2}}} + r(s)
\]
and $a_{p+1}(1) = g(1)$. Proposition \ref{prop.tau} now implies that
\[
\frac{1}{\#\P_v(T)} \sum_{\ell(p) < T} \left( \varphi_{\w{q}}(p)+ \ell(p)^{q/2} \right)^2  \sim \frac{g(1) T^q}{p}
\]
as $T\to\infty$ and so
\[
\frac{1}{\#\P_v(T)} \sum_{p \in \P_v(T)} \ell(p)^{q/2} \varphi_{\w{q}}(p)= o(T^q)
\]
as $T \to\infty$.
To conclude the proof we need to remove the weighting term $\ell(p)^{q/2}$. To do so we set
\[
\phi(t) = t^{-q/2} \ \text{ and } \ \wt{\pi}(t) = \sum_{p \in \P_v(t)} \ell(p)^{q/2} \ell(p)^{q/2} \varphi_{\w{q}}(p)
\]
for $t>0$ and note that by the above $\wt{\pi}(t)$ is $O(t^qe^t)$. It follows that
\[
\int_0^T \wt{\pi}(t) \phi'(t) \ dt = O( T \cdot T^q e^T T^{-1-q/2}) = O(T^{q/2} e^T)
\]
as $T\to\infty$.
However we also have that
\begin{align*}
\int_1^T \wt{\pi}(t) \phi'(t) \ dt &= \int_1^T\sum_{\ell(p) < T} \varphi_{\w{q}}(p) \ell(p)^{q/2} \phi'(t) \ dt + O(1) \\
&= \sum_{\ell(p) < T} \varphi_{\w{q}}(p) \ell(p)^{q/2} \int_{\ell(p)}^T \phi'(t)  dt\\
 &= T^{-q/2} \wt{\pi}(T) - \pi_{\w{q}}(T) + O(1)
\end{align*}
as $T\to\infty$. Rearranging this and using our estimates above gives the required result.
\end{proof}

We are now ready to deduce Theorem \ref{thm.mclt}

\begin{proof}[Proof of Theorem \ref{thm.mclt}]
To conclude the proof we apply the method of moments: \cite[Theorem 30.2]{billingsley}. Some justification is needed here. Indeed, the method of moments is usually used to prove one-dimensional central limit theorems for discrete sequences of distributions. However, by Theorem 29.4 in \cite{billingsley},  we can use the method of moments in higher dimensions. Lastly we note that it is possible to deduce the continuous limit theorem (i.e. as $T\to\infty$ through the reals) case from the discrete limit case as $n\to\infty$ through the natural numbers. We leave this simple deduction to the reader.
Finally we note that the non-degeneracy criteria follows from Lemma \ref{lem.secondderiv}.
\end{proof}

Before we move onto translation surfaces we prove Proposition \ref{prop.independent}.
\begin{proof}[Proof of Proposition \ref{prop.independent}]
We note that the mean $\Lambda$ and covariance matrix $\Sigma$ are obtained from the first and second derivatives of $\lambda(s,t)$. Since $\lambda(s,t)$ is independent of the choice of starting vertex, the result follows.
\end{proof}


\section{Translation surfaces}\label{sec.translation}

We recall the definitions from the introduction.
\begin{definition}
A translation surface $X$ is a compact surface with a flat metric
except at a finite set $\mathcal{X}= \{x_1, \cdots, x_n\}$ of singular points with cone angles $2\pi (k(x_i)+1)$, where $k(x_i) \in \mathbb N$, for $i=1, \ldots, n$.
\end{definition}

\begin{example}
This is a translation surface of genus $3$ with four singularities each with cone angle $4\pi $.

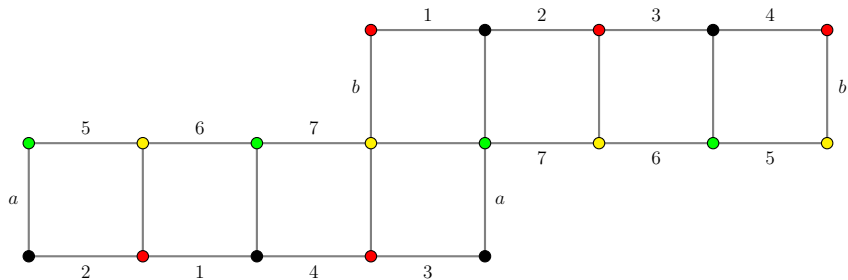
\begin{figure}[h!]
\begin{tikzpicture}[scale=0.5, every node/.style={scale=0.6}]
\draw[gray, thick] (0,0)--(12,0);
\draw[gray, thick] (0,3)--(21,3);
\draw[gray, thick] (9,6)--(21,6);
\draw[gray, thick] (0,0)--(0,3);
\draw[gray, thick] (3,0)--(3,3);
\draw[gray, thick] (6,0)--(6,3);
\draw[gray, thick] (9,0)--(9,3);
\draw[gray, thick] (12,0)--(12,3);
\draw[gray, thick] (15,3)--(15,6);
\draw[gray, thick] (18,3)--(18,6);
\draw[gray, thick] (21,3)--(21,6);
\draw[gray, thick] (12,3)--(12,6);
\draw[gray, thick] (9,3)--(9,6);
\draw [fill=black] (0,0) circle [radius=0.15];
\draw [fill=red] (3,0) circle [radius=0.15];
\draw [fill=black] (6,0) circle [radius=0.15];
\draw [fill=red] (9,0) circle [radius=0.15];
\draw [fill=black] (12,0) circle [radius=0.15];
\draw [fill=red] (9,6) circle [radius=0.15];
\draw [fill=black] (12,6) circle [radius=0.15];
\draw [fill=red] (15,6) circle [radius=0.15];
\draw [fill=black] (18,6) circle [radius=0.15];
\draw [fill=red] (21,6) circle [radius=0.15];
\draw [fill=green] (0,3) circle [radius=0.15];
\draw [fill=green] (6,3) circle [radius=0.15];
\draw [fill=green] (12,3) circle [radius=0.15];
\draw [fill=green] (18,3) circle [radius=0.15];
\draw [fill=yellow] (3,3) circle [radius=0.15];
\draw [fill=yellow] (9,3) circle [radius=0.15];
\draw [fill=yellow] (15,3) circle [radius=0.15];
\draw [fill=yellow] (21,3) circle [radius=0.15];
\node at (1.5,-0.4) {$2$};
\node at (4.5,-0.4) {$1$};
\node at (7.5,-0.4) {$4$};
\node at (10.5,-0.4) {$3$};
\node at (1.5,3.4) {$5$};
\node at (4.5,3.4) {$6$};
\node at (7.5,3.4) {$7$};
\node at (13.5,2.6) {$7$};
\node at (16.5,2.6) {$6$};
\node at (19.5,2.6) {$5$};
\node at (10.5,6.4) {$1$};
\node at (13.5,6.4) {$2$};
\node at (16.5,6.4) {$3$};
\node at (19.5,6.4) {$4$};
\node at (-0.4,1.5) {$a$};
\node at (12.4,1.5) {$a$};
\node at (8.6,4.5) {$b$};
\node at (21.4,4.5) {$b$};
\end{tikzpicture}
\caption{ A translation surface with $4$ singularities.}
\end{figure}
\end{example}

A path which does not pass through singularities is a locally distance minimising geodesic if it is a straight line segment. This includes local geodesics which start and end at singularities, known as {\it saddle connections}.  Geodesics can change direction if they go through a singular point.

More precisely, a locally minimising, length $R$ geodesic on a translation surface $X$ with singularity set $\mathcal{X}$ is a curve 
$\gamma:[0,R]\rightarrow X$ satisfying the following conditions:  
\begin{enumerate}[(i)]
\item there exist $0\leq t_1<...<t_n \leq R$, where $n\geq 0$, such that $\gamma(t_i)\in \mathcal{X}$;
\item for $t_i<t<t_{i+1}$ $\gamma(t)\in X\backslash \mathcal{X}$ for $i=1,\ldots, n-1$;
\item $\gamma:(t_i,t_{i+1}) \rightarrow X\backslash \mathcal{X}$ is a geodesic segment (possibly a saddle connection); and,
\item 
the smallest angle between $\gamma|_{(t_{i-1},t_i)}$ and $\gamma|_{(t_i,t_{i+1})}$ is at least $\pi$.
\end{enumerate}

Let ${\mathcal{S}}=\{s_1,s_2,...\}$ be the set of oriented saddle connections  ordered by non decreasing lengths.   

\begin{definition}
We define a saddle connection path $p
= (s_{i_1},...,s_{i_n})$  to be a finite string of oriented saddle collections $s_{i_1},...,s_{i_n}$ which form a local geodesic path.
\end{definition}

We denote by $\ell(p) = \ell(s_1)+\ell(s_2)+\cdots +\ell(s_n)$ the sum of the lengths of the constituent saddle connections.
We let $i(p), t(p) \in \mathcal{X}$
denote the initial and terminal singularities, respectively, of the 
saddle connection path  $p$. 

\begin{definition}
Suppose that $\varphi : \mathcal{S}  \to \R$ is a function that assigns to each singularity $s$ a real number $\varphi(s)$. Assume further that there exists $C > 0$ such that  $|\varphi(s)| \le C \ell(s)$ for all $s\in \mathcal{S}$. Then we extend $\varphi$ to saddle connection paths in the obvious way: if $p = (s_{i_1}, \ldots, s_{i_n})$ then we set
\[
\varphi(p) = \varphi(s_{i_1}) + \cdots + \varphi(s_{i_n}).
\]
We call such functions \textit{saddle cost functions}.
\end{definition}

To apply our result for graphs to translation surfaces we need analogues of conditions (G1), (G2), (G3) for translation surfaces: \\ 

\noindent \textbf{Translation Hypotheses}
\begin{enumerate}
\item[(T1)]
for all $t > 0$ we have $\sum_{s \in \mathcal{S}} e^{-t \ell(s)} < \infty$; 
\item[(T2)] there exists a constant $C >0$ such that for any directed saddle connections $s$ and $s'$ there exists a saddle connection path with length less than $C + \ell(s) + \ell(s')$ that begins with $s$ and ending with $s'$; and
\item[(T3)] there does not exist a
$d > 0$ such that 
\[
\{ \ell (c) : c \hbox{ is a closed saddle connection path} \} \subset d \mathbb N.
\]
\end{enumerate}
We claim that the above hypotheses hold for all translation surfaces.

\begin{proof}
Property $(\text{T}1)$ follows from the fact that the number of saddle connections of length less than $T$ grows quadratically in $T$ \cite{masur}.  The second property $(\text{T}2)$ was established in \cite[Proposition 3.11]{dankwart}. Lastly, $(\text{T}3)$ is an immediate corollary of Theorem \ref{thm.lengthspec} below.
\end{proof}

\subsection{Proof of main results}\label{sec.mainresults}

We are now ready to prove our main results.
Fix a translation surface $(X,\omega)$, $x\in \mathcal{X}$ and a vector of saddle cost functions $\overline{\varphi} : \P_x \to \R$. We form an infinite matrix $M$ with rows and columns indexed by $\mathcal{S}$:
\[
M(s,s')=\begin{cases} 1 &\mbox{if }  t(s)=i(s'), \\
0 & \mbox{otherwise. }  \end{cases} 
\]
Then for each  $u\in \mathbb C$ and $v \in \C^d$ 	
we define the perturbed matrix $M_{u,v}$ by 
\[
M_{u,v}(s,s')=M(s,s')e^{-u\ell(s') - \langle v , \overline{\varphi}(s') \rangle}
\]
 for $s,s'\in \mathcal{S}$.

The translation hypotheses above are direct analogues of the graph properties $(\text{T}1), (\text{T}2), (\text{T}3)$ from Section \ref{sec.graph} and we can follow the arguments presented in that section.

\begin{proof}[Proof of Theorem \ref{thm.general} and Theorem \ref{thm.generalmulti}]
The results follow directly from the arguments in Section \ref{sec.graph} using the matrices $M_{u,v}$.
\end{proof}

We can then deduce our corollary for number of singularities visited.

\begin{proof}[Proof of Theorem \ref{thm.sing}]
We can apply Theorem \ref{thm.general} to deduce that the large deviation and central limit theorem hold for the singularity length $|p|$. We just need to show that the variance for the central limit theorem is strictly positive. Suppose it is not, this would imply that there exists $\tau > 0$ such that $|p|= \tau \ell(p)$ for all closed paths $p$. This would imply that $\ell(p) \in \frac{1}{\tau} \Z$ for all $p$ contradicting $(\text{T}3)$.
\end{proof}

We also have the following result. Recall that $|\cdot|_\theta$ denotes the angle change saddle cost function defined in the introduction.
\begin{theorem}\label{thm.angle}
Let $(X,\omega)$ be a translation surface and fix a singularity $x \in\mathcal{X}$. Let $|\cdot|_\theta$ and $\ell(\cdot)$ denote the angle change and geometric length of a saddle connection path respectively. Then there exists $\Lambda > 0$ such that for any $\epsilon > 0$ 
\[
\limsup_{T\to\infty} \frac{1}{T} \log \left( \mu_T \left\{ p \in \P_x : \left| \frac{|p|_\theta}{T} - \Lambda \right| > \epsilon \right\}\right) < 0.
\]
Furthermore there exists $\sigma^2 > 0$ such that for any $a,b \in \R, a<b$
\[
\lim_{T\to\infty} \mu_T\left\{ p\in \P_x: \frac{|p|_\theta - \Lambda T}{\sqrt{T}} \in [a,b] \right\} = \frac{1}{\sqrt{2 \pi} \sigma}\int_a^b e^{-t^2/2\sigma^2} \ dt.
\]
\end{theorem}
\begin{proof}
As for the proof of Theorem \ref{thm.sing} we just need to verify the non-degeneracy criteria. To see this, note that given a singularity $x \in \mathcal{X}$, a direction $v$ emanating from $x$ and $\epsilon >0$, we can find a saddle connection from $x$ to another singularity that leaves $x$  at an angle that is within $\epsilon$ of the direction of $v$. Using this fact and $(\text{T}2)$ it is easy to see how to construct closed saddle paths for which $|p|_\theta/\ell(p)$ is arbitrarily small. This means that there cannot be $\tau > 0$ such that $|p|_\theta = \tau \ell(p)$ for all closed $p$.
\end{proof}

\begin{proof}[Proof of Corollary \ref{cor.holonomy}]
Applying Theorem \ref{thm.generalmulti} we deduce that a 2-dimensional central limit theorem holds for the holonomy vector. 
The mean for this central limit theorem is the zero vector by symmetry. Indeed, if we write $\overline{\Lambda} \in \R^2$ for the mean then we have that
\[
\overline{\Lambda} = \lim_{T\to\infty} \frac{1}{\#\P_x(T)} \sum_{p \in\P_x(T)} \frac{1}{T} \int_p \omega.
\]
However, for each $p \in \P_x(T)$ there is $p' \in \P_x(T)$ which is the path $p$ with the opposite orientation. Therefore
\[
\overline{\Lambda} =  \lim_{T\to\infty} \frac{1}{\#\P_x(T)} \sum_{p \in\P_x(T)} \frac{1}{T} \int_{p'} \omega =  \lim_{T\to\infty} \frac{1}{\#\P_x(T)} \sum_{p \in\P_x(T)} - \frac{1}{T} \int_p \omega = -\overline{\Lambda}.
\]
To conclude the proof we therefore need to show that the central limit theorem is non-degenerate. Suppose that it is, then there exists $v \in \R^2$ such that
\[
\left\langle v , \int_p \omega \right\rangle = 0 \ \text{ for all closed saddle paths $p$.}
\]
This would, by $(\text{T}2)$, imply that there exists $C >0$ such that for any saddle connection $s$ we have
\[
\left|\left\langle v , \int_s \omega \right\rangle\right| \le C. 
\]
However the holonomy vectors have angles that are dense in the circle and so this is not possible.
\end{proof}

\begin{proof}[Proof of Corollary \ref{cor.sing}]
As in the previous proof, we can immediately deduce that a multi-dimensional central limit theorem holds. In this case the mean is a strictly positive vector $\Lambda(Y) \in \R^k_{>0}$. If the central limit theorem is degenerate  then there exist $v \in \R^k$ such that
\[
\left\langle v, |p|_Y - \Lambda(Y)\ell(p) \right\rangle = 0 \ \text{ for all closed saddle paths $p$,}
\]
or $\langle v, |p|_Y \rangle = \langle v, \Lambda(Y) \rangle \ell(p)$ for all closed $p$. 
Suppose that $v$ is non-zero and it's $j$th component $v_j \neq 0$. If $s$ is a saddle connection starting and ending at $y_j$ then the above expression implies that $2 v_j =  v_j\Lambda(Y)_j \ell(s)$, i.e. every saddle connection from $y_j$ to itself has the same length. This is absurd and so $v$ must be the zero vector.
\end{proof}


\subsection{A general positive variance criteria}\label{sec.variance}
We would like a way of verifying whether a general class of saddle cost functions satisfy a non-degenerate central limit theorem. To do this we study the number theoretic properties of saddle connections. Recall that a collection of numbers $p_1,\ldots, p_N \in \R$ are rationally independent if $\sum_{j=1}^N a_j p_j = 0$ for $a_1,\ldots, a_N \in \Q$ implies that $a_1 = \cdots = a_N = 0$.
\begin{definition}
The geometric length spectrum of the translation surface $(X,\omega)$ is
\[
\mathcal{L}_{(X,\omega)} = \{ \ell(p) : p\in \P_x \ \text{for some $x \in \mathcal{X}$ and $p$ corresponds to a closed saddle path}\} .
\]
The saddle length spectrum is $\ell(\mathcal{S}) = \{ \ell(s) : s \in \mathcal{S}\}.$
\end{definition}
We need the following result. 
\begin{proposition} \label{prop.irrat}
Suppose that $p(x) = x^2 + ax + b$ is a quadratic polynomial with $a,b \in \R$ with distinct roots. Then the set
\[
\{ \sqrt{p(k)} : k \in \mathbb{Z}_{\ge 0} \}
\]
contains arbitrarily large subsets of rationally independent numbers.
\end{proposition}
We are very grateful to Peter M\"uller who suggested the following proof of this proposition \cite{Muller}.
\begin{proof}
Suppose for contradiction that the conclusion is wrong. Then there are $x_1, \ldots, x_n \in \R$ such that $\sqrt{p(k)}$ is in the rational span of these $x_i$ for all $k$ large enough. Now let $R$ be the ring generated by the $x_i$ and $a,b$ and write $K$ for the quotient field of $R$. We then have that $\sqrt{p(k)} \in K$ for all large $k$. By replacing $R$ by $R[1/a]$ for some $a \in \mathbb{Z}$ we can assume that the integral closure of $R$ in $K$ is integrally closed and finitely generated \cite[Proposition 4.1]{serge}.

We now note that, since $p(x)$ has distinct roots, we can find $m \in \Z$ such that $p(x^2 + m)$ has distinct roots. Then the curve $Y^2 - p(x^2 + m)$ is an elliptic curve $E$. Since $R$ is finitely generated $E(R)$ (i.e. the points on $E$ that have coordinates in $R$) is finite by \cite[Theorem 2.4 in $\S 8$]{serge}. Since $\Z$ is contained in $R$ and $R$ is integrally closed, this is a contradiction: by our assumption $\sqrt{p(k^2 + m)} \in R$ for all $k$ large enough implying that $(k, \sqrt{p(k^2 + m)}) \in E$ for all large $k$.
\end{proof}
We then have the following.
\begin{proposition}\label{thm.lengthspec}
Let $(X,\omega)$ be a translation surface (with singularities). Then both the geometric length spectrum and the saddle length spectrum contain arbitrarily large subsets of numbers that are rationally independent, i.e. for any $N \ge 1$ there exists $l_1, \ldots l_N \in L_{(X,\omega)}$ such that the only solution $\alpha_1, \ldots, \alpha_n \in \Q$ such that 
\[
\sum_{j=1}^N \alpha_j l_j = 0
\]
is the solution $\alpha_1 = \cdots = \alpha_N = 0$ (and similarly for $\ell(\mathcal{S})$).
\end{proposition}
\begin{proof}
We claim that $\mathcal{L}_{(X,\omega)}$ contains arbitrarily large subsets of rationally independent numbers if and only if $\ell(\mathcal{S})$ contains arbitrarily large subsets of rationally independent numbers. Indeed, by $(\text{T}2)$, for any saddle connection $s$ we can find closed saddle paths $p_1,p_2$ such that $p_1 s p_2 s'$ is a closed saddle path where $s'$ is $s$ with reversed orientation. We then have that $2\ell(s) = \ell(p_1) + \ell(p_2) + \ell(p_1 s p_2 s')$ and the claim follows. We therefore just need to verify the theorem for $\ell(\mathcal{S})$.

We now recall that translation surfaces contain embedded cylinders: a cylinder such that the boundaries circles are saddle connections. We can then find $h,w > 0$ and $\theta \in (0,\pi)$ such that $\ell(\mathcal{S})$ contains the lengths
\[
\sqrt{ h^2 + (kw)^2 - 2k hw \cos(\theta) } \ \ \text{ for $k \ge 1$.}
\]
Indeed, the number $\sqrt{ h^2 + (kw)^2 - 2k hw \cos(\theta) } $ is the last side lengths of the triangle with a side of length $h$ that meets a side of length $kw$ at angle $\theta$.

To conclude the proof of the theorem we apply Proposition \ref{prop.irrat} to the polynomial
\[
p(x) = x^2 - \left( \frac{2h \cos(\theta)}{w}\right) \, x  + \frac{h^2}{w^2}
\]
which we note has distinct roots.
\end{proof} 

\begin{remark}
The above proof actually shows that $\mathcal{L}_{(X,\omega)}$ (and also $\ell(\mathcal{S})$) is not contained in a ring generated by finitely many real numbers.
\end{remark}

We then deduce the following.

\begin{corollary}
Let $(X, \omega)$ be a translation surface and fix $x \in \mathcal{X}$. Let $\ell(\cdot)$ denote the  geometric length and let $\varphi(\cdot)$ be a saddle cost function on saddle connection paths. Suppose further that $\varphi$ is positive, non-zero and that there exist $\alpha_1,\ldots, \alpha_n \in \R$ such that 
\[
\{ \varphi(p) : \text{$p$ is a closed saddle path}\} \subseteq \alpha_1 \Q \oplus \alpha_2 \Q \oplus \cdots \oplus \alpha_n\Q.
\]
Then there exist $\Lambda >0$ and $\sigma^2 > 0$ such that for any $a,b \in \R, a<b$
\[
\lim_{T\to\infty} \mu_T\left\{ p\in \P_x: \frac{\varphi(p) - \Lambda T}{\sqrt{T}} \in [a,b] \right\} = \frac{1}{\sqrt{2 \pi} \sigma}\int_a^b e^{-t^2/2\sigma^2} \ dt.
\]
\end{corollary}
\begin{proof}
We just need to verify the non-degeneracy. Suppose that the variance is $0$. This would imply that for any closed saddle path $p$, $\varphi(p) = \Lambda \ell(p)$. However this would imply that $L_{(X,\omega)}$ does not contain arbitrarily large subsets of rationally independent numbers contradicting Proposition \ref{thm.lengthspec}.
\end{proof}

\begin{example}
The above corollary implies that there is a non-degenerate central limit theorem for the saddle cost functions
\[
s \mapsto |\mathfrak{R}_s|, s\mapsto  |\mathfrak{I}_s| \ \ \ \text{ where  $\mathfrak{R}_s$ and $ \mathfrak{I}_s$ are the real and imaginary change along $s$ (as in the introduction).}
\]
The same is true for positive linear combinations of these functions. This follows from the fact that the relative homology $H_1(X,\mathcal{X}, \mathbb{Z})$ has finite rank.
\end{example}

\bibliographystyle{alpha}
\bibliography{CLT}

\end{document}